\documentclass{amsart}

\usepackage{amssymb}

\usepackage{amsmath}
\usepackage{amsthm}
\usepackage[usenames]{color}
\usepackage{eepic,epic}
\usepackage{graphicx}
\usepackage{amssymb}
\usepackage{mathtools}
\usepackage{tikz}
\usepackage{psfrag}
\usepackage{xcolor}

\newtheorem{thm}{Theorem}[section]

\newtheorem{lem}[thm]{Lemma}
\newtheorem{prop}[thm]{Proposition}

\theoremstyle{definition}
\newtheorem{defn}[thm]{Definition}
\theoremstyle{remark}

\numberwithin{equation}{section}

\newcommand{\N}{\mathbb N}
\newcommand{\M}{\mathbb M}

\newcommand{\R}{\mathbb R}

\newcommand{\dist}{\operatorname{dist}}

\begin{document}



\title[Adhesion and volume filling in 1-D population dynamics]{Adhesion and volume filling in\\one-dimensional population dynamics \\under no-flux boundary condition}

\author{Hyung Jun Choi}
\address{School of Liberal Arts, Korea University of Technology and Education, Cheonan 31253, Republic of Korea}
\curraddr{}
\email{hjchoi@koreatech.ac.kr}
\thanks{}

\author{Seonghak Kim}
\address{Department of Mathematics, College of Natural Sciences, Kyungpook National University, Daegu 41566, Republic of Korea}
\curraddr{}
\email{shkim17@knu.ac.kr}
\thanks{}

\author{Youngwoo Koh}
\address{Department of Mathematics Education, Kongju National University, Kongju 32588, Republic of Korea}
\curraddr{}
\email{ywkoh@kongju.ac.kr}
\thanks{}



\date{}

\dedicatory{}

\begin{abstract}
We study the (generalized) one-dimensional population model developed by Anguige \& Schmeiser \cite{AS}, which reflects cell-cell adhesion and volume filling under no-flux boundary condition. In this generalized model, depending on the adhesion and volume filling parameters $\alpha,\beta\in[0,1],$ the resulting equation is classified into six types.
Among these, we focus on the type exhibiting strong effects of both adhesion and volume filling, which results in a class of advection-diffusion equations of the forward-backward-forward type. For five distinct cases of initial maximum, minimum and average population densities, we derive the corresponding patterns for the global behavior of weak solutions to the initial and no-flux boundary value problem.  Due to the presence of a negative diffusion regime, we indeed prove that the problem is ill-posed and admits infinitely many global-in-time weak solutions, with the exception of one specific case of the initial datum. This nonuniqueness is inherent in the method of convex integration that we use to solve the Dirichlet problem of a partial differential inclusion arising from the ill-posed problem.
\end{abstract}

\maketitle

\tableofcontents

\section{Introduction}

When individuals correlate their movements with others, spontaneous aggregations emerge in various natural systems, such as flocks, swarms and herds \cite{JRD,WBTB}; bacteria and cancer cells \cite{ASKIB,ME}; and human groupings in cities and towns \cite{YZX}.

To understand the collective movement of individuals in response to the presence of others, advection-diffusion equations with nonlocal advection terms have been extensively studied in contexts such as cell sorting behaviors \cite{APS}; locust swarms \cite{BT}; home range formation in wolves \cite{BLP}; and human pedestrian traffic \cite{CGL}. More recently, a local continuum model of thin-film type equations was derived in \cite{FBC} from a general nonlocal model in the limit of short-range interactions, aiming to observe cell sorting patterns and explore the relationship between model parameters and the differential adhesion hypothesis.

While many modern biological aggregation models rely on nonlocal continuum approaches, discrete-space models have also been historically utilized (see \cite{APS,FBC} for references).
In particular, to account for an inhomogeneous distribution of cells that can move randomly to some extent but are also influenced by cell-cell adhesion and volume filling, {Anguige \& Schmeiser} \cite{AS} proposed a one-dimensional discrete-space model in continuous time for linear stability analysis and steady-state solutions. They also examined the continuum limit of their discrete-space model as the mesh size $h\to 0^+$, which leads to the initial-Neumann boundary value problem of a quasilinear advection-diffusion equation:
\begin{equation}\label{ibp-second}
\begin{cases}
  u_t=(\rho(u))_{xx} & \mbox{in $\Omega\times(0,\infty)$}, \\
  u=u_0 & \mbox{on $\Omega\times\{t=0\}$}, \\
  u_x(0,t)=u_x(1,t)=0 & \mbox{for $t>0$},
\end{cases}
\end{equation}
where $u=u(x,t)$ denotes the population density of a single species at position $x$ and time $t$, $\Omega=(0,1)\subset\R$ is a favorable habitat of unit size, $u_0=u_0(x)$ is an initial population density, and
the flux function $\rho$ is given by
\begin{equation}\label{rho-function-AS}
\rho(s)=\alpha s^3-2\alpha s^2+s\quad(s\in\R),
\end{equation}
with the adhesion parameter $\alpha\in [0,1]$.
Here, the Neumann (or reflecting) boundary condition should imply that the total population remains constant over time.

In the model described in \cite{AS}, if the species is highly adhesive, meaning $\alpha>3/4$ in (\ref{rho-function-AS}), the diffusivity $\sigma(s)=\rho'(s)$ can be negative over a nonempty open interval. This results in the equation $u_t=(\rho(u))_{xx}$ being \emph{backward parabolic} for population density values $s=u$ within that interval. Consequently, if the initial population density $u_0$ falls within this backward regime, problem (\ref{ibp-second}) may be ill-posed. This issue led the authors of \cite{AS} to anticipate certain pattern formations in the population density $u$ over time. However, due to the inapplicability of standard parabolic theory methods, no existence results were available for such a highly adhesive continuum model until the authors of this paper investigated, in \cite{CKK}, the existence, nonuniqueness and global behaviors of weak solutions to problem (\ref{ibp-second}) with the Neumann boundary condition replaced by the Dirichlet boundary condition.

In fact, the work \cite{CKK} reformulates the equation in problem (\ref{ibp-second})(\ref{rho-function-AS}) using a generalized flux function:
\begin{equation}\label{rho-function-CKK}
\rho(s)=\alpha\beta s^3-2\alpha s^2+s\quad(s\in\R),
\end{equation}
where $\alpha,\beta\in [0,1]$ are the adhesion and volume filling parameters, respectively. The case with the strongest volume filling effect, $\beta=1$, is studied in \cite{AS}. In this generalized model, the diffusivity $\sigma(s)=\rho'(s)$ is given by
\[
\sigma(s)=3\alpha\beta s^2-4\alpha s+1\quad(s\in\R).
\]
Here, the compact set $[0,1]^2$ of $(\alpha,\beta)$ is divided into six disjoint subsets, each corresponding to a distinct type of problem (\ref{ibp-second})(\ref{rho-function-CKK}) (see Figure \ref{fig1}). For example, (FDBDF) denotes the case where, for some $0<s^-_0<s^+_0<1$, the diffusivity  $\sigma(s)$ is positive (forward parabolic) on $[0,s^-_0)\cup(s^+_0,1]$, negative (backward parabolic) on $(s^-_0,s^+_0)$, and zero (degenerate) at $s=s^\pm_0.$ The case (F) falls within the realm of standard parabolic theory \cite{Ln}. The cases (FDF) and (FD) can be addressed using the theory of degenerate parabolic equations \cite{Di}. The remaining three cases, (FDB), (FDBD) and (FDBDF), which involve a backward diffusion regime, can be managed using  the convex integration method \cite{MSv2}, as demonstrated in our previous work \cite{CKK} for the Dirichlet boundary condition and in this paper for the no-flux boundary condition.

\begin{figure}[ht]
\begin{center}
\begin{tikzpicture}[scale =0.9]
    \draw[->] (-0.5,0) -- (5.5,0);
    \draw[->] (0,-0.5) -- (0,5.5);
    \draw[thick] (0,0) -- (5,0);
    \draw[thick] (5,0) -- (5,5);
    \draw[thick] (5,5) -- (0,5);
    \draw[thick] (0,5) -- (0,0);
    \draw[dashed] (0,0) -- (5/2,10/3);
    \draw[dashed] (0,10/3) -- (5/2,10/3);
    \draw[dashed] (5/2,10/3) -- (5/2,0);
    \draw[thick] (5/2,10/3) -- (15/4,5);
    \draw[thick] (5/4,0) .. controls (2,3.95) and  (3.7,4.5) .. (5,5);
    \draw (0,5.5) node[left] {$\beta$};
    \draw (5.5,0) node[below] {$\alpha$};
    \draw (0,5) node[left] {$1$};
    \draw (5,0) node[below] {$1$};
    \draw (0,10/3) node[left] {$\frac{2}{3}$};
    \draw (5/2,0) node[below] {$\frac{1}{2}$};
    \draw (2.5/2,0) node[below] {$\frac{1}{4}$};
    \draw (15/4,5) node[above] {$\beta=\frac{4}{3}\alpha$};
    \draw (5.9,5) node[above] {$\beta=-\frac{1}{3\alpha}+\frac{4}{3}$};
    \draw (1,5/2) node[] {(F)};
    \draw (3.5,5/2) node[] {(FDB)};
    \draw[dashed,->] (4,4.8) -- (5.5,4.8);
    \draw[dashed,->] (3.55,4.3) -- (5.5,4.3);
    \draw[dashed,->] (3.2,4.3) -- (2,5.2);
    \draw[dashed,->] (1.5,1.1) -- (1.5,-0.6);
    \draw (6.5,4.8) node[] {(FDBDF)};
    \draw (6.5,4.3) node[] {(FDBD)};
    \draw (2,5.5) node[] {(FDF)};
    \draw (1.5,-0.9) node[] {(FD)};

    \end{tikzpicture}
\end{center}
\caption{Classification of problem (\ref{ibp-second})(\ref{rho-function-CKK}).}
\label{fig1}
\end{figure}


In this paper, we focus on the generalized model in \cite{CKK}, which originated from \cite{AS}, and consider a specific type of the resulting equation that yields a class of advection-diffusion equations of the forward-backward-forward type. Specifically, we analyze the case (FDBDF)  with the flux function (\ref{rho-function-CKK}) (see Figure \ref{fig2}),where
\[
\frac{2}{3}<\beta\le 1\quad\mbox{and}\quad\frac{3}{4}\beta<\alpha<\frac{1}{4-3\beta}.
\]
In this scenario, there are exactly two numbers $0<s^-_0<s^+_0<1$ such that $\sigma(s^\pm_0)=\rho'(s^\pm_0)=0.$
We then study the existence and asymptotic behaviors of global-in-time weak solutions to the initial and no-flux boundary value problem:
\begin{equation}\label{ib-P-intro}
\begin{cases}
  u_t=(\rho(u))_{xx} & \mbox{in $\Omega\times(0,\infty)$}, \\
  u=u_0 & \mbox{on $\Omega\times\{t=0\}$}, \\
  \sigma(u)u_x=0 & \mbox{on $\partial\Omega\times(0,\infty)$},
\end{cases}
\end{equation}
where $\Omega=(0,L)\subset\R$ is a habitat of size $L>0$.

The main result of this paper, in a simplified form,  is as follows.

\begin{thm}\label{thm:main-simplified}
Let $u_0\in C^{2+a}(\bar\Omega;[0,1])$ for some $0<a<1,$ and assume $u'_0(0)=u'_0(L)=0.$ If $u_0(\bar\Omega)\cap [s^-_0,s^+_0]\ne\emptyset,$ then
problem \eqref{ib-P-intro} admits infinitely many global weak solutions $u\in L^\infty(\Omega\times(0,\infty);[0,1])$, all of which conserve the initial total population for all time.
\end{thm}

This theorem will be detailed through four theorems: Theorems \ref{thm:(ii)-1}, \ref{thm:(ii)-2}, \ref{thm:(iii)-1} and \ref{thm:(iv)}, each reflecting different global behaviors of solutions depending on the initial population density $u_0.$


\underline{\textbf{Plan of the paper:}} In the remainder of this section, we introduce the notations that will be used throughout the paper and derive a natural definition of a global weak solution to problem (\ref{ib-P-intro}). Section \ref{sec:preliminaries} gathers some preliminary results from  classical parabolic theory. The main result of the paper, Theorem \ref{thm:main-simplified}, is developed in section \ref{sec:mainresult}, along with a strategy for obtaining the detailed results, Theorems \ref{thm:(ii)-1}, \ref{thm:(ii)-2}, \ref{thm:(iii)-1} and \ref{thm:(iv)}. Section \ref{sec:generic} is devoted to formulating a differential inclusion problem and presenting its special solvability result as Theorem \ref{thm:two-wall}, which is a crucial component for proving  Theorems \ref{thm:(ii)-1}, \ref{thm:(ii)-2}, \ref{thm:(iii)-1} and \ref{thm:(iv)} in section \ref{sec:proof_main}.

\subsection{Notations} Hereafter, let $L>0$ denote the size of a favorable habitat $\Omega=(0,L)\subset\R$.
For $\tau\in(0,\infty],$ we write
\[
\Omega_\tau:=\Omega\times(0,\tau)\subset\R^2.
\]

Let $k$, $m$ and $n$ be positive integers. Let $0<a<1$ and $U\subset\R^n$ be an open set.
\begin{itemize}
\item[(i)] We denote by $C^k(U)$ the space of functions $u:U\to\R$ whose partial derivatives of order up to $k$ exist and are continuous in $U$.
\item[(ii)] Let $C^k(\bar{U})$ be the space of functions $u\in C^k(U)$ whose partial derivatives of order up to $k$ are uniformly continuous in every bounded subset of $U$.
\item[(iii)] Let $n=2$; so $U\subset\R^2=\R_x\times \R_t.$ We define $C^{2,1}(U)$ to be the space of functions $u:U\to\R$ such that $u_x$, $u_{xx}$ and $u_t$ exist and are continuous in $U$. Also, let $C^{2,1}(\bar{U})$ denote the space of functions $u\in C^{2,1}(U)$ such that $u_x$, $u_{xx}$ and $u_t$ are uniformly continuous in every bounded subset of $U$.
\item[(iv)] We denote by $C^{2+a}(\bar\Omega)$ the space of functions $u\in C^{2}(\bar\Omega)$ with
\[
\sup_{x,y\in\Omega,\,x\ne y}\frac{|u_{xx}(x)-u_{xx}(y)|}{|x-y|^a}<\infty.
\]
\item[(v)] Let $0<T<\infty.$  We define $C^{2+a,1+\frac{a}{2}}(\bar\Omega_T)$ to be the space of functions $u\in C^{2,1}(\bar\Omega_T)$ whose quantities
\[
\sup_{x,y\in\Omega,\,x\ne y,\,0<t<T}\frac{|u_{xx}(x,t)-u_{xx}(y,t)|}{|x-y|^a},\; \sup_{x\in\Omega,\,0<s,t<T,\,s\ne t}\frac{|u_{xx}(x,s)-u_{xx}(x,t)|}{|s-t|^\frac{a}{2}},
\]
\[
\sup_{x,y\in\Omega,\,x\ne y,\,0<t<T}\frac{|u_{t}(x,t)-u_{t}(y,t)|}{|x-y|^a},\; \sup_{x\in\Omega,\,0<s,t<T,\,s\ne t}\frac{|u_{t}(x,s)-u_{t}(x,t)|}{|s-t|^\frac{a}{2}}
\]
are all finite.
\item[(vi)] We denote by $\M^{m\times n}$  the space of $m\times n$ real matrices.
\item[(vii)] For a Lebesgue measurable set $E\subset\R^n,$ its $n$-dimensional measure is denoted by $|E|=|E|_n$ with subscript $n$ omitted if it is clear from the context.
\item[(viii)] We use  $W^{k,p}(U)$ to denote the space of functions $u\in L^{p}(U)$ whose weak partial derivatives of order up to $k$ exist in $U$ and belong to $L^{p}(U)$.
\end{itemize}

\subsection{Global weak solutions}
To derive a natural definition of a weak solution to problem (\ref{ib-P-intro}), let $u_0\in C^2(\bar\Omega;[0,1])$ be such that
\[
\sigma(u_0)u_0'=0\quad\mbox{on $\partial\Omega=\{0,L\}$}.
\]

Assume that $u\in C^{2,1}(\bar\Omega_\infty;[0,1])$ is a global classical solution to problem (\ref{ib-P-intro}). Fix any $T>0$, and choose a test function $\varphi\in C^\infty(\bar\Omega\times [0,T])$ such that
\[
\varphi=0\;\;\mbox{on $\bar\Omega\times\{t=T\}$}\;\;\mbox{and}\;\;\varphi_x=0\;\;\mbox{on $\partial\Omega\times[0,T]$}.
\]
Then from the integration by parts,
\[
\begin{split}
0= &\, \int_0^T\int_0^L (u_t -(\rho(u))_{xx})\varphi\,dxdt \\
= &\, \int_0^L (u(x,T)\varphi(x,T)- u(x,0)\varphi(x,0))\,dx - \int_0^T\int_0^L u\varphi_t\,dxdt \\
&\, - \int_0^T (\sigma(u(L,t))u_x(L,t)\varphi(L,t)-\sigma(u(0,t))u_x(0,t)\varphi(0,t))\,dt\\
&\, + \int_0^T\int_0^L (\rho(u))_x\varphi_x\,dxdt \\
= &\, \int_0^L (u(x,T)\varphi(x,T)- u(x,0)\varphi(x,0))\,dx - \int_0^T\int_0^L u\varphi_t\,dxdt \\
&\, - \int_0^T (\sigma(u(L,t))u_x(L,t)\varphi(L,t)-\sigma(u(0,t))u_x(0,t)\varphi(0,t))\,dt\\
&\, + \int_0^T (\rho(u(L,t))\varphi_{x}(L,t)-\rho(u(0,t))\varphi_{x}(0,t))\,dt\\
&\, - \int_0^T\int_0^L \rho(u)\varphi_{xx}\,dxdt \\
= &\, -\int_0^L u_0(x)\varphi(x,0)\,dx - \int_0^T\int_0^L (u\varphi_t +\rho(u)\varphi_{xx})\,dxdt.
\end{split}
\]

Conversely, assume that $u\in C^{2,1}(\bar\Omega_\infty;[0,1])$ is a function satisfying that
\[
\int_0^T\int_0^L (u\varphi_t +\rho(u)\varphi_{xx})\,dxdt+\int_0^L u_0(x)\varphi(x,0)\,dx=0
\]
for all $T>0$ and $\varphi\in C^\infty(\bar\Omega\times [0,T])$ with
\[
\varphi=0\;\;\mbox{on $\bar\Omega\times\{t=T\}$}\;\;\mbox{and}\;\;\varphi_x=0\;\;\mbox{on $\partial\Omega\times[0,T]$}.
\]
We will check below that $u$ is a global classical solution to problem (\ref{ib-P-intro}). 

To show that the first of (\ref{ib-P-intro}) holds, fix any $\varphi\in C^\infty_c(\Omega_\infty).$ Choose a $T=T_\varphi>0$ so large that $\mathrm{spt}(\varphi)\subset\subset \Omega\times(0,T).$ Then from the integration by parts,
\[
\begin{split}
0= &\, \int_0^T\int_0^L (u\varphi_t +\rho(u)\varphi_{xx})\,dxdt+\int_0^L u_0(x)\varphi(x,0)\,dx \\
= &\, \int_0^T\int_0^L (-u_t +(\rho(u))_{xx})\varphi\,dxdt=\int_0^\infty\int_0^L (-u_t +(\rho(u))_{xx})\varphi\,dxdt.
\end{split}
\]
Thus, the first of (\ref{ib-P-intro}) is satisfied.

Next, to check that the second of (\ref{ib-P-intro}) holds, fix any $\psi\in C^\infty_c(\Omega)$. Choose a function $\omega\in C^\infty(\R)$ such that
\[
\omega=1\;\;\mbox{on $(-\infty,0]$}\;\;\mbox{and}\;\;\omega=0\;\;\mbox{on $[2,\infty)$},
\]
and define $\varphi(x,t)=\psi(x)\omega(t)$ for $(x,t)\in\bar{\Omega}_\infty.$
Then with $T=2,$ it follows from the first of (\ref{ib-P-intro}) that
\[
\begin{split}
0= &\, \int_0^2\int_0^L (u\varphi_t +\rho(u)\varphi_{xx})\,dxdt+\int_0^L u_0(x)\varphi(x,0)\,dx \\
= &\, \int_0^L (u(x,2)\varphi(x,2)-u(x,0)\varphi(x,0))\,dx -\int_0^2\int_0^L u_t\varphi\,dxdt \\
&\, +\int_0^2 (\rho(u(L,t))\varphi_{x}(L,t)- \rho(u(0,t))\varphi_{x}(0,t))\,dt\\
&\, -\int_0^2 ( (\rho(u))_x (L,t)\varphi(L,t)-(\rho(u))_x (0,t)\varphi(0,t)) \,dt\\
&\, +\int_0^2\int_0^L (\rho(u))_{xx}\varphi\,dxdt  +\int_0^L u_0(x)\varphi(x,0)\,dx\\
= &\, \int_0^L (u_0(x)-u(x,0))\psi(x)\,dx.
\end{split}
\]
Thus, the second of (\ref{ib-P-intro}) is true.

Finally, to see that the third of (\ref{ib-P-intro}) holds, fix any $\omega\in C^\infty_c(0,\infty).$ Let $T=T_{\omega}>0$ be chosen so large that $\mathrm{spt}(\omega)\subset\subset(0,T)$.
Choose two functions $\psi_0,\psi_1\in C^\infty(\R)$ such that
\[
\psi_0(x)=\begin{cases}
            1 & \mbox{for $x\le\frac{1}{4}L$} \\[2mm]
            0 & \mbox{for $x\ge\frac{3}{4}L$}
          \end{cases}
          \;\;\mbox{and}\;\;
\psi_1(x)=\begin{cases}
            0 & \mbox{for $x\le\frac{1}{4}L$} \\[2mm]
            1 & \mbox{for $x\ge\frac{3}{4}L$}.
          \end{cases}
\]
For $i=0,1$, define $\varphi_i(x,t)=\psi_i(x)\omega(t)$ for $(x,t)\in\bar\Omega_\infty.$
Then we have from the first of (\ref{ib-P-intro}) that for $i=0,1$,
\[
\begin{split}
0=  &\, \int_0^L (u(x,T)\varphi_i(x,T)-u(x,0)\varphi_i(x,0))\,dx -\int_0^T\int_0^L u_t\varphi_i\,dxdt \\
&\, +\int_0^T (\rho(u(L,t))(\varphi_i)_{x}(L,t)- \rho(u(0,t))(\varphi_i)_{x}(0,t))\,dt\\
&\, -\int_0^T ( (\rho(u))_x (L,t)\varphi_i(L,t)-(\rho(u))_x (0,t)\varphi_i(0,t))\,dt\\
&\, +\int_0^T\int_0^L (\rho(u))_{xx}\varphi_i\,dxdt  +\int_0^L u_0(x)\varphi_i(x,0)\,dx\\
= &\, (-1)^i\int_0^T (\rho(u))_x (iL,t)\omega(t)\,dt=(-1)^i\int_0^\infty (\rho(u))_x (iL,t)\omega(t)\,dt;
\end{split}
\]
that is, $\sigma(u(x,t))u_x(x,t)=(\rho(u))_x (x,t)=0$ for all $(x,t)\in\partial\Omega\times(0,\infty)$. Hence, the third of (\ref{ib-P-intro}) follows.

Summarizing the previous discussion, we have the following.

\begin{prop}\label{prop:classical-global}
Let $u_0\in C^2(\bar\Omega;[0,1])$ satisfy the compatibility condition,
\[
\sigma(u_0)u_0'=0\quad\mbox{on $\partial\Omega$},
\]
and $u\in C^{2,1}(\bar\Omega_\infty;[0,1]).$
Then $u$ is a global classical solution to problem \eqref{ib-P-intro} if and only if
\[
\int_0^T\int_0^L (u\varphi_t +\rho(u)\varphi_{xx})\,dxdt+\int_0^L u_0(x)\varphi(x,0)\,dx=0
\]
for all $T>0$ and $\varphi\in C^\infty(\bar\Omega\times [0,T])$ with
\[
\varphi=0\;\;\mbox{on $\bar\Omega\times\{t=T\}$}\;\;\mbox{and}\;\;\varphi_x=0\;\;\mbox{on $\partial\Omega\times[0,T]$}.
\]
\end{prop}

Motivated by this observation, we fix the definition of a global weak solution to problem (\ref{ib-P-intro}) as follows.

\begin{defn}\label{def:global-weak-sol}
Let $u_0\in L^\infty(\Omega;[0,1])$ and $u\in L^\infty(\Omega_\infty;[0,1]).$
Then we say that $u$ is a \emph{global weak solution} to problem (\ref{ib-P-intro}) provided that
for all $T>0$ and $\varphi\in C^\infty(\bar\Omega\times [0,T])$ with
\[
\varphi=0\;\;\mbox{on $\bar\Omega\times\{t=T\}$}\;\;\mbox{and}\;\;\varphi_x=0\;\;\mbox{on $\partial\Omega\times[0,T]$},
\]
one has
\begin{equation}\label{eq:weaksol}
\int_0^T\int_0^L (u\varphi_t +\rho(u)\varphi_{xx})\,dxdt+\int_0^L u_0(x)\varphi(x,0)\,dx=0.
\end{equation}
\end{defn}

\section{Preliminaries}\label{sec:preliminaries}
In this section, we consider the initial and no-flux boundary value problem:
\begin{equation}\label{ib-P-pre}
\begin{cases}
  u_t=(\rho(u))_{xx} & \mbox{in $\Omega_\infty$,} \\
  u=u_0 & \mbox{on $\Omega\times\{t=0\}$,} \\
  \sigma(u)u_x=0 & \mbox{on $\partial\Omega\times(0,\infty)$},
\end{cases}
\end{equation}
where
\begin{equation}\label{assume-sigma-1}
\rho\in C^2(\R)\;\;\mbox{and}\;\;\sigma:=\rho'>0\;\;\mbox{in $\R$}.
\end{equation}
In the following, we gather some useful properties that any classical solution to problem (\ref{ib-P-pre}) should satisfy.

\begin{lem}\label{lem:classical}
Let $u\in C^{2,1}(\bar\Omega_\infty)$ be a solution to problem \eqref{ib-P-pre}, where the initial function $u_0\in C^2(\bar\Omega)$ satisfies the compatibility condition:
\[
\sigma(u(0))u_0'(0)=\sigma(u(L))u_0'(L)=0.
\]
Then
\begin{itemize}
\item[(i)]
$
\min_{\bar\Omega}u_0\le u\le\max_{\bar\Omega}u_0\;\;\mbox{in $\Omega_\infty$},
$
\item[(ii)]
$
\frac{1}{L}\int_\Omega u(x,t)\,dx= \bar{u}_0\;\;\mbox{for all $t>0$},
$
\item[(iii)]
$
\|u(\cdot,t)-\bar{u}_0\|_{L^\infty(\Omega)}\to 0\;\;\mbox{as $t\to\infty$},
$
\end{itemize}
where $\bar{u}_0:=\frac{1}{L}\int_\Omega u_0(x)\,dx.$
\end{lem}

Although this lemma may be standard, we provide its proof for the readers' convenience.

\begin{proof}
(i) For simplicity, let us write $M_0:=\max_{\bar\Omega}u_0$ and $m_0:=\min_{\bar\Omega}u_0$. We have to show that
\[
m_0\le u\le M_0\;\;\mbox{on $\bar\Omega_\infty$}.
\]

We first verify that $u\ge m_0$ on $\bar\Omega_\infty$.
To do so, fix any $\epsilon>0,$ and define
\begin{equation}\label{lem:classical-0}
v(x,t)=v_\epsilon(x,t)=u(x,t)+\epsilon t\;\;\forall (x,t)\in\bar\Omega_\infty.
\end{equation}
We claim that
\begin{equation}\label{lem:classical-claim}
v(x,t)\ge m_0\;\;\forall (x,t)\in\bar\Omega_\infty.
\end{equation}
If this is the case, then for each $(x,t)\in\bar\Omega_\infty,$ letting $\epsilon\to 0^+$, it follows that
$
u(x,t)\ge m_0.
$

We now prove inequality (\ref{lem:classical-claim}) by contradiction. Suppose on the contrary that there exists a point $(x_0,t_0)\in\bar\Omega_\infty$ with
\[
v(x_0,t_0)<m_0.
\]
Fix a number $T\in (t_0,\infty).$ Then
\begin{equation}\label{lem:classical-1}
v(x_1,t_1)=\min_{\bar\Omega_T} v\le v(x_0,t_0)<m_0
\end{equation}
for some $(x_1,t_1)\in\bar\Omega_T$. If $t_1=0,$ then from (\ref{ib-P-pre}) and (\ref{lem:classical-0}),
\[
m_0\le u_0(x_1)=v(x_1,0)=v(x_1,t_1)<m_0,
\]
a contradiction. Thus, $0<t_1\le T$ so that
\begin{equation}\label{lem:classical-2}
v_t(x_1,t_1)\le 0.
\end{equation}
If $x_1\in\partial\Omega,$ then from (\ref{ib-P-pre}) and (\ref{lem:classical-0}), $v_x(x_1,t_1)=u_x(x_1,t_1)=0$ as $\sigma(u(x_1,t_1))>0;$ thus, $v_{xx}(x_1,t_1)\ge0$. If $x_1\in\Omega,$ then from (\ref{lem:classical-1}), we have $v_x(x_1,t_1)=0$  and $v_{xx}(x_1,t_1)\ge0$. In any case, from (\ref{ib-P-pre}), (\ref{assume-sigma-1}) and (\ref{lem:classical-0}), at the point $(x,t)=(x_1,t_1)$,
\[
v_t=u_t+\epsilon = \rho'(u)u_{xx}+\rho''(u)u_x^2+\epsilon= \sigma(u)v_{xx}+\epsilon\ge \epsilon>0,
\]
which is a contradiction to (\ref{lem:classical-2}). Therefore, (\ref{lem:classical-claim}) holds.

To show that $u\le M_0$ on $\bar\Omega_\infty$, for any fixed $\epsilon>0$, introduce
\[
w(x,t)=w_\epsilon(x,t)=u(x,t)-\epsilon t\;\;\forall (x,t)\in\bar\Omega_\infty.
\]
In a similar manner as above,  one can show that $w\le M_0$ on $\bar\Omega_\infty$. Then it follows immediately that $u\le M_0$ on $\bar\Omega_\infty$. We omit the details.

(ii) From (\ref{ib-P-pre}), we have
\[
\frac{d}{dt}\int_\Omega u(x,t)\,dx= \int_\Omega u_t(x,t)\,dx =\int_\Omega (\sigma(u)u_x)_{x}(x,t)\,dx=0
\]
for all $t>0.$ Thus, by continuity,
\[
\frac{1}{L}\int_\Omega u(x,t)\,dx= \bar{u}_0\;\;\mbox{for all $t>0$}.
\]

(iii) From (\ref{ib-P-pre}), the integration by parts, (i) and Poincar\'e's inequality, we have
\[
\begin{split}
\frac{d}{dt}\int_\Omega (u-\bar{u}_0)^2\,dx = &\,  2\int_\Omega (u-\bar{u}_0) (\sigma(u)u_x)_{x}\,dx =-2\int_\Omega \sigma(u)u_x^2\,dx \\
\le &\, -2 s_0\int_\Omega u_x^2\,dx \le -2 s_0 C \int_\Omega (u-\bar{u}_0)^2\,dx
\end{split}
\]
for all $t>0$, where $s_0:=\min_{[m_0,M_0]}\sigma>0$, and $C>0$ is a constant, depending only on $L$. Thus, from Gronwall's inequality, we have
\[
\|u(\cdot,t)-\bar{u}_0\|_{L^2(\Omega)}\le \|u_0-\bar{u}_0\|_{L^2(\Omega)} e^{-s_0 Ct}\to 0\;\;\mbox{as $t\to\infty.$}
\]
Appealing to \cite[Theorem 1.1]{GS}, we now conclude that
\[
\|u(\cdot,t)-\bar{u}_0\|_{L^\infty(\Omega)}\to 0\;\;\mbox{as $t\to\infty.$}
\]
\end{proof}

\section{Main results}\label{sec:mainresult}
We now study problem (\ref{ib-P-intro}); that is, the initial and no-flux boundary value problem in one space dimension:
\begin{equation}\label{ib-P}
\begin{cases}
u_{t} =(\sigma(u)u_x)_x=(\rho(u))_{xx} & \mbox{in $\Omega_\infty$,}\\
u =u_0 & \mbox{on $\Omega\times \{t=0\}$},\\
\sigma(u)u_x=0 & \mbox{on $\partial\Omega\times(0,\infty)$,}
\end{cases}
\end{equation}
where $u_0:\Omega\to [0,1]$ is a given initial population density, $u(x,t)\in[0,1]$ represents the population density at a space-time point $(x,t)\in\Omega_\infty$, the diffusivity $\sigma:\R\to\R$ is given by
\[
\sigma(s)=3\alpha\beta s^2- 4\alpha s +1\;\;(s\in\R)
\]
for some constants $\frac{2}{3}<\beta\le 1$ and $\frac{3}{4}\beta<\alpha<\frac{1}{4-3\beta}$, and
\[
\rho(s)=\alpha\beta s^3-2\alpha s^2+s \;\;(s\in\R).
\]

Letting
\[
s^\pm_{0}=\frac{2\alpha \pm \sqrt{4\alpha^2 -3\alpha\beta}}{3\alpha\beta},
\]
we observe that $0<s^-_0<s^+_0<1$ and that
\begin{equation}\label{sigma-1}
\sigma(s) \begin{cases}
                    >0 & \mbox{for $s\in [0,s^-_0)\cup(s^+_0,1]$,} \\
                    <0 & \mbox{for $s^-_0<s<s^+_0$,} \\
                    =0 & \mbox{for $s=s^\pm_0$.}
                  \end{cases}
\end{equation}
It is also easily checked that
\begin{equation}\label{ineq-rho}
\rho(s)>0\;\;\forall s\in(0,1].
\end{equation}

Let
\[
r^*=\min\{\rho(s^-_0),\rho(1)\};
\]
then from (\ref{sigma-1}) and (\ref{ineq-rho}), we have
\[
r^*>\rho(s^+_0)>0.
\]
For each $r\in[\rho(s^+_0),r^*],$ let $s^+(r)\in[s^+_0,1]$ and $s^-(r)\in(0,s^-_0]$ denote the unique numbers with
\[
\rho(s^\pm(r))=r.
\]
Let us write
\[
s^\pm_1:=s^\pm(\rho(s^+_0))\;\;\mbox{and}\;\; s^\pm_2:=s^\pm(r^*);
\]
then
\[
0<s^-_1<s^-_2\le s^-_0<s^+_0 =s^+_1<s^+_2\le 1,
\]
and
\[
s^-_2= s^-_0\;\;\mbox{or}\;\; s^+_2= 1.
\]
Observe also that
\[
\begin{cases}
  \rho(s^-_0)\le\rho(1) & \mbox{if $\frac{3}{4}\beta<\alpha\le-\frac{4\beta}{3\beta^2-4\beta-4}$,} \\
  \rho(s^-_0)>\rho(1) & \mbox{if $-\frac{4\beta}{3\beta^2-4\beta-4}<\alpha<\frac{1}{4-3\beta}$.}
\end{cases}
\]
We say that  problem (\ref{ib-P}) is of \emph{type (I)} and \emph{type (II)} if $\rho(s^-_0)\le\rho(1)$ and $\rho(s^-_0)>\rho(1)$, respectively (see Figure \ref{fig2}).

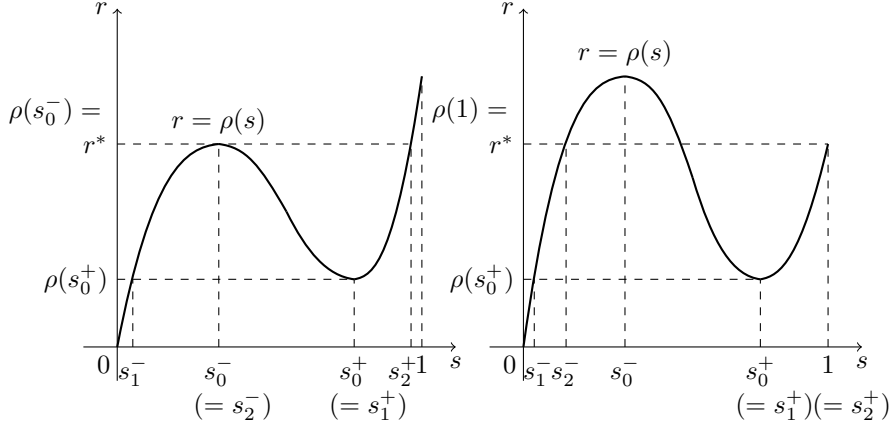
\begin{figure}[ht]
\begin{center}
\begin{tikzpicture}[scale =0.9]
    \draw[->] (-0.5,0) -- (5,0);
    \draw[->] (5.5,0) -- (11,0);
    \draw[->] (0,-0.5) -- (0,5);
    \draw[->] (6,-0.5) -- (6,5);
 \draw[dashed] (6,1)--(9.5,1);
 \draw[dashed] (6,3)--(10.5,3);
 \draw[dashed] (7.5,0)--(7.5,4);
 \draw[dashed] (9.5,0)--(9.5,1);
 \draw[dashed] (10.5,0)--(10.5,3);
 \draw[dashed] (6.16,0)--(6.16,1);
 \draw[dashed] (6.63,0)--(6.63,3);
 \draw (7.5, 0) node[below] {$s^-_0$};
 \draw (9.5, 0) node[below] {$s^+_0$};
 \draw (9.71, -0.5) node[below] {$(=s^+_1)$};
 \draw (10.85, -0.5) node[below] {$(=s^+_2)$};
 \draw (6, 3) node[left] {$r^*$};
 \draw (5.9, 3.5) node[left] {$\rho(1)=$};
 \draw (6, 1) node[left] {$\rho(s^+_0)$};
 \draw (6.22, 0) node[below] {$s^-_1$};
 \draw (6.63, 0) node[below] {$s^-_2$};
 \draw[dashed] (0,1)--(3.5,1);
 \draw[dashed] (0,3)--(4.3,3);
 \draw[dashed] (1.5,0)--(1.5,3);
 \draw[dashed] (3.5,0)--(3.5,1);
 \draw[dashed] (4.5,0)--(4.5,4);
 \draw[dashed] (0.23,0)--(0.23,1);
 \draw[dashed] (4.34,0)--(4.34,3);
 \draw (1.5, 0) node[below] {$s^-_0$};
 \draw (3.5, 0) node[below] {$s^+_0$};
 \draw (3.71, -0.5) node[below] {$(=s^+_1)$};
 \draw (1.71, -0.5) node[below] {$(=s^-_2)$};
 \draw (0, 3) node[left] {$r^*$};
 \draw (-0.1, 3.5) node[left] {$\rho(s^-_0)=$};
 \draw (0, 1) node[left] {$\rho(s^+_0)$};
 \draw (0.23, 0) node[below] {$s^-_1$};
	\draw[thick]   (6, 0) .. controls (6.5,3.6) and  (7,3.95)   ..(7.5,4);
    \draw[thick]   (7.5,4) .. controls (7.9,3.95) and  (8.1,3.7)   ..(8.5,2.5);
    \draw[thick]   (8.5,2.5) .. controls (8.8,1.5) and  (9.1,1.05)   ..(9.5,1);
	\draw[thick]   (9.5,1) .. controls  (9.8,1.05) and (10.1,1.3) ..(10.5,3);
    \draw[thick]   (0, 0) .. controls (0.5,2.6) and  (1,2.95)   ..(1.5,3);
    \draw[thick]   (1.5,3) .. controls (1.9,2.95) and  (2.1,2.7)   ..(2.5,2);
    \draw[thick]   (2.5,2) .. controls (2.8,1.4) and  (3.1,1.05)   ..(3.5,1);
	\draw[thick]   (3.5,1) .. controls  (3.8,1.05) and (4.1,1.3) ..(4.5,4);
	\draw (11,0) node[below] {$s$};
    \draw (5,0) node[below] {$s$};
    \draw (7.5, 4) node[above] {$r=\rho(s)$};
    \draw (1.5, 3) node[above] {$r=\rho(s)$};
   \draw (4.5, 0) node[below] {$1$};
   \draw (10.5, 0) node[below] {$1$};
   \draw (4.22, 0) node[below] {$s^+_2$};
   \draw (0, 5) node[left] {$r$};
   \draw (6, 5) node[left] {$r$};
   \draw (-0.2, 0) node[below] {$0$};
   \draw (5.8, 0) node[below] {$0$};
    \end{tikzpicture}
\end{center}
\caption{Graphs of type (I) and type (II) of  $r=\rho(s)$ $(0\le s\le 1)$.}
\label{fig2}
\end{figure}



\subsection{Initial population density}
We assume that the initial population density $u_0$ to problem (\ref{ib-P}) fulfills the regularity condition:
\[
u_0\in C^{2+a}(\bar\Omega;[0,1])\;\;\mbox{for some $0<a<1$}
\]
and the compatibility condition:
\[
u_0'(0)=u_0'(L)=0.
\]
Let us write
\[
M_0:=\max_{\bar{\Omega}}u_0,\;\;m_0:=\min_{\bar{\Omega}}u_0,\;\;\mbox{and}\;\; \bar{u}_0:=\frac{1}{L}\int_\Omega u_0(x)\,dx;
\]
then $0\le m_0\le \bar{u}_0 \le M_0\le 1.$ So we have one and only one of the following four cases.
\begin{itemize}
\item[(i)] $M_0<s^-_0$ or $m_0>s^+_0$;
\item[(ii)] $\bar{u}_0<s^-_0\le M_0$;
\item[(iii)] $m_0\le s^+_0<\bar{u}_0$;
\item[(iv)] $s^-_0\le\bar{u}_0\le s^+_0$.
\end{itemize}
We also divide case (ii) into the following two subcases.
\begin{itemize}
\item[(ii-1)] $\bar{u}_0<s^-_2\le s^-_0\le M_0$;
\item[(ii-2)] problem (\ref{ib-P}) is of type (II), and $s^-_2\le\bar{u}_0< s^-_0\le M_0$.
\end{itemize}

\subsection{Main results}
We now present our main results below.
All the results in this subsection will be proved in section \ref{sec:proof_main}.

The first result, which pertains to case (i), is not a part of Theorem \ref{thm:main-simplified} and is based solely on classical parabolic theory.

\begin{thm}[Case (i): smooth stabilization]\label{thm:(i)} Assume that
\[
\mbox{\emph{(i)} $M_0<s^-_0$ or $m_0>s^+_0$.}
\]
Then there exists a global classical solution $u\in C^{2,1}(\bar\Omega_\infty;[0,1])$ to problem \eqref{ib-P} with
\[
u\in C^{2+a,1+\frac{a}{2}}(\bar\Omega_T;[0,1])\;\;\forall T>0
\]
satisfying the following:
\begin{itemize}
\item[(a)] \underline{\emph{Maximum principle:}}
\[
0\le\min_{\bar\Omega}u(\cdot,s)\le u(x,t)\le\max_{\bar\Omega}u(\cdot,s)\le 1\;\;\forall x\in\bar\Omega,\,\forall t\ge s\ge0;
\]
\item[(b)] \underline{\emph{Conservation of total population:}}
\[
\int_\Omega u(x,t)\,dx=L \bar{u}_0\;\;\mbox{for all $t\ge0$};
\]
\item[(c)] \underline{\emph{Uniform stabilization:}}
\[
\|u(\cdot,t)-\bar{u}_0\|_{L^\infty(\Omega)}\to 0\;\;\mbox{as $t\to\infty$}.
\]
\end{itemize}
\end{thm}

The result indicates that if the initial population density distribution $u_0(\bar\Omega)$ does not overlap with the backward regime $[s^-_0,s^+_0]$, then the smooth global solution $u$ is distributed within the forward regime $[0,s^-_0)\cup(s^+_0,1]$ for all times and stabilizes to the initial average density $\bar u_0$ as time approaches infinity.

The second result, concerning  case (ii-1), is as follows.


\begin{thm}[Case (ii-1): finite extinction of density mixtures]\label{thm:(ii)-1}
Assume that
\[
\mbox{\emph{(ii-1)} $\bar{u}_0<s^-_2\le s^-_0\le M_0$}.
\]
Let
\[
\max\{\rho(\bar{u}_0),\rho(s_1^-)\}<r_1<r_2\le r^*=\rho(s^-_2).
\]
Then there exist a function $u^\star\in C^{2,1}(\bar\Omega_\infty;[0,1])$ with
\[
u^\star\in C^{2+a,1+\frac{a}{2}}(\bar\Omega_T;[0,1])\;\;\forall T>0,
\]
a nonempty bounded open set $Q\subset\Omega_\infty$ with $\bar{Q}\cap (\Omega\times\{0\})\ne\emptyset$,
and infinitely many global weak solutions $u\in L^{\infty}(\Omega_\infty;[0,1])$ to problem \eqref{ib-P} satisfying the following:
\begin{itemize}
\item[(a)] \underline{\emph{Smoothing in finite time:}}
\[
u=u^\star\;\;\mbox{in $\Omega_\infty\setminus \bar{Q};$}
\]
\item[(b)] \underline{\emph{Density mixtures:}}
\[
u\in[s^-(r_1),s^-(r_2)]\cup[s^+(r_1),s^+(r_2)]\;\;\mbox{a.e. in $Q;$}
\]
\item[(c)] \underline{\emph{Fine-scale oscillations:}} for any nonempty open set $O\subset Q,$
\[
\underset{O}{\mathrm{ess\,osc}}\, u:=\underset{O}{\mathrm{ess\,sup}}\, u-\underset{O}{\mathrm{ess\,inf}}\, u \ge s^+(r_1)-s^-(r_2)>0;
\]
\item[(d)] \underline{\emph{Conservation of total population:}}
\[
\int_\Omega u(x,t)\,dx= \int_\Omega u^\star(x,t)\,dx=L\bar{u}_0\;\;\forall t\ge 0;
\]
\item[(e)] \underline{\emph{Maximum principle:}}
\[
0\le\min_{\bar\Omega}u^\star(\cdot,s)\le u^\star(x,t)\le\max_{\bar\Omega}u^\star(\cdot,s)\le 1\;\;\forall x\in\bar\Omega,\,\forall t\ge s\ge0;
\]
\item[(f)] \underline{\emph{Uniform stabilization:}}
\[
\|u^\star(\cdot,t)-\bar{u}_0\|_{L^\infty(\Omega)}\to 0\;\;\mbox{as $t\to\infty$.}
\]
\end{itemize}
\end{thm}

This result shows that if the initial maximum density $M_0$ is at least $s^-_0$ but the initial average density $\bar u_0$ is less than $s^-_2(\le s^-_0),$ then fine-scale density mixtures will appear immediately over a finite time interval. However, the weak solutions will eventually become smooth and stabilize to $\bar u_0$ after a certain period.

The third result for case (ii-2) addresses problem (\ref{ib-P}) of type (II) where the initial maximum density $M_0$ is at least $s^-_0$, but the initial average density $\bar u_0$ falls between $s^-_2$ (inclusive) and $s^-_0$ with $s^-_2$.


\begin{thm}[Case (ii-2): anomalous patching in type (II) for $\bar{u}_0\in[s^-_2,s^-_0)$]\label{thm:(ii)-2}
Assume that problem \eqref{ib-P} is of type (II) and that
\[
s^-_2\le\bar{u}_0<s^-_0\le M_0.
\]
Let
\[
\rho(s^-_1)\le r_1<r_2\le r^*=\rho(1).
\]
Then there exist a function $u^\star_1\in C^{2+a,1+\frac{a}{2}}(\bar\Omega_{T_1};[0,1])$ for some $T_1>0$, a function $u^\star_2\in C^{2,1}(\bar\Omega\times[T_1,\infty);[0,1])$ with
\[
u^\star_2\in C^{2+a,1+\frac{a}{2}}(\bar\Omega\times[T_1,T];[0,1])\;\;\forall T>T_1 \;\;\mbox{and}\;\; u_2^\star=u^\star_1\;\;\mbox{on $\bar{\Omega}\times\{t=T_1\}$},
\]
a nonempty bounded open set $Q\subset\Omega_{T_1}$ with $\bar{Q}\cap (\Omega\times\{0\})\ne\emptyset$,
and infinitely many global weak solutions $u\in L^{\infty}(\Omega_\infty;[0,1])$ to problem \eqref{ib-P} satisfying the following:
\begin{itemize}
\item[(a)] \underline{\emph{Smoothing in finite time:}}
\[
u=u^\star_1\;\;\mbox{in $\Omega_{T_1}\setminus\bar{Q}$}\;\;\mbox{and}\;\; u=u^\star_2\;\;\mbox{on $\bar\Omega\times[T_1,\infty);$}
\]
\item[(b)] \underline{\emph{Density mixtures:}}
\[
u\in[s^-(r_1),s^-(r_2)]\cup[s^+(r_1),s^+(r_2)]\;\;\mbox{a.e. in $Q;$}
\]
\item[(c)] \underline{\emph{Fine-scale oscillations:}} for any nonempty open set $O\subset Q,$
\[
\underset{O}{\mathrm{ess\,osc}}\, u \ge s^+(r_1)-s^-(r_2)>0;
\]
\item[(d)] \underline{\emph{Conservation of total population:}}
\[
\int_\Omega u(x,t)\,dx= \int_\Omega (\chi_{[0,T_1)}(t) u^\star_1(x,t)+\chi_{[T_1,\infty)}(t) u^\star_2(x,t))\,dx=L\bar{u}_0\;\;\forall t\ge 0;
\]
\item[(e)] \underline{\emph{Maximum principle:}}
\[
m_0\le\min_{\bar\Omega}u^\star_2(\cdot,s)\le u^\star_2(x,t)\le\max_{\bar\Omega}u^\star_2(\cdot,s)\le M_0\;\;\forall x\in\bar\Omega,\,\forall t\ge s\ge T_1;
\]
\item[(f)] \underline{\emph{Uniform stabilization:}}
\[
\|u^\star_2(\cdot,t)-\bar{u}_0\|_{L^\infty(\Omega)}\to 0\;\;\mbox{as $t\to\infty$.}
\]
\end{itemize}
\end{thm}

In this case, we perform a somewhat artificial adjustment using the smooth functions $u^\star_1$ and $u^\star_2$, which share the same values at $t=T_1$, to obtain global weak solutions. These solutions are partially equal to $u^\star_1$ in $\Omega_{T_1}$ and identically equal to $u^\star_2$ for $t\ge T_1$. Fine density mixtures of the weak solutions appear elsewhere in the finite space-time domain $\Omega_{T_1}$, where they differ from $u^\star_1.$ However, the global behavior of the weak solutions is identical to that of $u^\star_2.$

The fourth result, related to case (iii), is as follows.

\begin{thm}[Case (iii): finite extinction of density mixtures]\label{thm:(iii)-1}
Assume that
\[
\mbox{\emph{(iii)} $m_0\le s^+_0<\bar{u}_0$}.
\]
Let
\[
\rho(s^+_0)\le r_1<r_2<\min\{\rho(\bar{u}_0),r^*\}.
\]
Then there exist a function $u^\star\in C^{2,1}(\bar\Omega_\infty;[0,1])$ with
\[
u^\star\in C^{2+a,1+\frac{a}{2}}(\bar\Omega_T;[0,1])\;\;\forall T>0,
\]
a nonempty bounded open set $Q\subset\Omega_\infty$ with $\bar{Q}\cap (\Omega\times\{0\})\ne\emptyset$,
and infinitely many global weak solutions $u\in L^{\infty}(\Omega_\infty;[0,1])$ to problem \eqref{ib-P} satisfying the following:
\begin{itemize}
\item[(a)] \underline{\emph{Smoothing in finite time:}}
\[
u=u^\star\;\;\mbox{in $\Omega_\infty\setminus \bar{Q};$}
\]
\item[(b)] \underline{\emph{Density mixtures:}}
\[
u\in[s^-(r_1),s^-(r_2)]\cup[s^+(r_1),s^+(r_2)]\;\;\mbox{a.e. in $Q;$}
\]
\item[(c)] \underline{\emph{Fine-scale oscillations:}} for any nonempty open set $O\subset Q,$
\[
\underset{O}{\mathrm{ess\,osc}}\, u \ge s^+(r_1)-s^-(r_2)>0;
\]
\item[(d)] \underline{\emph{Conservation of total population:}}
\[
\int_\Omega u(x,t)\,dx= \int_\Omega u^\star(x,t)\,dx=L\bar{u}_0\;\;\forall t\ge 0;
\]
\item[(e)] \underline{\emph{Maximum principle:}}
\[
0\le\min_{\bar\Omega}u^\star(\cdot,s)\le u^\star(x,t)\le\max_{\bar\Omega}u^\star(\cdot,s)\le 1\;\;\forall x\in\bar\Omega,\,\forall t\ge s\ge0;
\]
\item[(f)] \underline{\emph{Uniform stabilization:}}
\[
\|u^\star(\cdot,t)-\bar{u}_0\|_{L^\infty(\Omega)}\to 0\;\;\mbox{as $t\to\infty$.}
\]
\end{itemize}
\end{thm}

This result is quite similar to the result of case (ii-1), involving local fine density mixtures and global stabilization.

The final main result of the paper addresses the most intriguing situation, case (iv), where the initial average density $\bar u_0$ lies in the backward regime $[s^-_0, s^+_0]$.


\begin{thm}[Case (iv): everlasting density mixtures]\label{thm:(iv)}
Assume that
\[
\mbox{\emph{(iv)} $\bar{u}_0\in[s^-_0, s^+_0]$,}
\]
and let
\[
\rho(s^+_0)< r_0 < r^*.
\]
Then there exist a function $u^\star_1\in C^{2+a,1+\frac{a}{2}}(\bar\Omega_{T_1};[0,1])$ for some $T_1>0$, a nonempty open set $Q_1\subset\Omega_{T_1}$ with
\[
\bar{Q}_1\cap(\Omega\times\{0\})\ne\emptyset\;\;\mbox{and} \;\;\bar\Omega\times\{T_1\}\subset\bar{Q}_1,
\]
and infinitely many global weak solutions $u\in L^{\infty}(\Omega_\infty;[0,1])$ to problem \eqref{ib-P} satisfying the following:
\begin{itemize}
\item[(a)] \underline{\emph{Smoothness in $\Omega_{T_1}\setminus\bar Q_1$:}}
\[
u=u^\star_1\;\;\mbox{in $\Omega_{T_1}\setminus\bar Q_1;$}
\]
\item[(b)] \underline{\emph{Local and global density mixtures:}}
\[
u\in[s^-_1,s^-_2]\cup[s^+_1,s^+_2]\;\;\mbox{a.e. in $Q_1\cup(\Omega\times(T_1,\infty));$}
\]
\item[(c)] \underline{\emph{Local and global fine-scale oscillations:}} for any nonempty open set $O\subset Q_1\cup(\Omega\times(T_1,\infty)),$
\[
\underset{O}{\mathrm{ess\,osc}}\, u \ge s^+_0-s^-_0>0;
\]
\item[(d)] \underline{\emph{Conservation of total population:}}
\[
\int_\Omega u(x,t)\,dx=L\bar{u}_0\;\;\forall t\ge 0;
\]
\item[(e)] \underline{\emph{Eventual two-point density:}}
\[
\|\dist(u,\{s^+(r_0),s^-(r_0)\})\|_{L^\infty(\Omega\times(t,\infty))}\to 0\;\;\mbox{as $t\to\infty$.}
\]
\end{itemize}
\end{thm}

In this situation, since the initial average density $\bar u_0$ falls within the backward regime $[s^-_0, s^+_0]$, it is inevitable that the weak solutions will always experience fine density mixtures. These density mixtures may emerge immediately after the initial time $t=0$. Subsequently, from a certain moment $t=T_1$ onward, density mixtures will be present throughout the entire habitat $\Omega$.

\subsection{Approach by differential inclusion} Let us take a moment here to explain our approach to prove Theorems \ref{thm:(ii)-1}, \ref{thm:(ii)-2}, \ref{thm:(iii)-1} and \ref{thm:(iv)}.
To solve the equation in (\ref{ib-P}), we formally put $v_x=u$ in $\Omega_\infty$ for some function $v:\Omega_\infty\to\R$; so we consider the equation,
\[
v_t=(\rho(v_x))_x\;\;\mbox{in $\Omega_\infty$.}
\]

To solve the previous equation in the sense of distributions in $\Omega_\infty,$
we may try to find a vector function $z=(v,w)\in W^{1,\infty}(\Omega_\infty;\R^{1+1})$ with $v_x\in L^\infty(\Omega_\infty;[0,1])$ such that
\begin{equation}\label{alt-system}
w_x=v\;\;\mbox{and}\;\;w_t=\rho(v_x)\;\;\mbox{a.e. in $\Omega_\infty$}.
\end{equation}
If there is such a function $z=(v,w),$ we take $u=v_x \in L^\infty(\Omega_\infty;[0,1])$; then from the integration by parts, for each $\varphi\in C^\infty_c(\Omega_\infty)$,
\[
\begin{split}
\int_0^\infty\int_0^L  & (u\varphi_t +\rho(u)\varphi_{xx})\,dxdt = \int_0^{T_\varphi}\int_0^L (v_x\varphi_t+\rho(v_x)\varphi_{xx})\,dxdt \\
 = &\, \int_0^{T_\varphi}\int_0^L (-v\varphi_{tx}+w_t\varphi_{xx})\,dxdt = \int_0^{T_\varphi}\int_0^L (-v\varphi_{tx}+w_x\varphi_{xt})\,dxdt=0,
\end{split}
\]
where $T_\varphi:=\sup_{(x,t)\in\mathrm{spt}(\varphi)}t +1$. Hence, $u$ is a global weak solution of the equation in (\ref{ib-P}) in the sense of distributions in $\Omega_\infty.$

On the other hand, for each $b\in\R,$ define
\[
\Sigma(b)=\Sigma_{\rho}(b)=\left\{
\begin{pmatrix}
s & c \\
b & \rho(s)
\end{pmatrix}\in\M^{2\times 2}\,\Big|\, 0\le s\le 1,\, c\in\R
\right\};
\]
then system (\ref{alt-system}) is equivalent to the inhomogeneous partial differential inclusion,
\[
\nabla z=\begin{pmatrix}
v_x & v_t \\
w_x & w_t
\end{pmatrix}\in \Sigma(v)\;\;\mbox{a.e. in $\Omega_\infty$},
\]
where $\nabla=(\partial_x,\partial_t)$ is the space-time gradient operator. In this regard, utilizing the method of convex integration by M\"uller \& \v Sver\'ak \cite{MSv2}, we aim at solving this inclusion for certain sets $K(b)\subset\Sigma(b)$ $(b\in\R)$ in a generic setup (section \ref{sec:generic}) while reflecting the initial and no-flux boundary conditions in (\ref{ib-P}).

\section{Generic problem}\label{sec:generic}

In this independent section, we present a generic inclusion problem that can be applied to the main problem (\ref{ib-P}) as a special case.
Specifically, we will use Theorem \ref{thm:two-wall} to prove Theorems \ref{thm:(ii)-1}, \ref{thm:(ii)-2}, \ref{thm:(iii)-1}, and \ref{thm:(iv)}.
A detailed proof of Theorem \ref{thm:two-wall} can be found in \cite{CKK}.

\subsection{Two-wall inclusions}

As a setup, we fix some generic notations and introduce a two-wall partial differential inclusion of inhomogeneous type.

\subsubsection{Related sets}

Let $r_1<r_2$, and let $\omega_1,\omega_2\in C([r_1,r_2])$ be any two functions such that
\[
\max_{[r_1,r_2]}\omega_1 < \min_{[r_1,r_2]}\omega_2.
\]
For each $b\in\R$, define the matrix sets
\[
\begin{split}
K^+(b)=K^+_{\omega_2}(b) & =\bigg\{\begin{pmatrix}
\omega_2(r) & c \\
b & r
\end{pmatrix}\in\M^{2\times 2}\, \Big|\, r\in[r_1,r_2],\, c\in\R \bigg\},\\
K^-(b)=K^-_{\omega_1}(b) & =\bigg\{\begin{pmatrix}
\omega_1(r) & c \\
b & r
\end{pmatrix}\in\M^{2\times 2}\, \Big|\, r\in[r_1,r_2],\, c\in\R \bigg\},\\
U(b)=U_{\omega_1,\omega_2}(b) & =\bigg\{\begin{pmatrix}
s & c \\
b & r
\end{pmatrix}\in\M^{2\times 2}\, \Big|\, r\in(r_1,r_2),\, \omega_1(r)<s<\omega_2(r), \, c\in\R \bigg\},\\
\end{split}
\]
and $K(b)=K_{\omega_1,\omega_2}(b)=K^+(b)\cup K^-(b)$. Let
\[
\begin{split}
K^+=K^+_{\omega_2} & =\{(\omega_2(r),r)\,|\, r\in[r_1,r_2]\}, \\
K^-=K^-_{\omega_1} & =\{(\omega_1(r),r)\,|\, r\in[r_1,r_2]\},
\end{split}
\]
and $K=K_{\omega_1,\omega_2}=K^+\cup K^-$. Also, let
\[
U =\{(s,r)\, |\, r\in(r_1,r_2),\, \omega_1(r)<s<\omega_2(r) \}.
\]

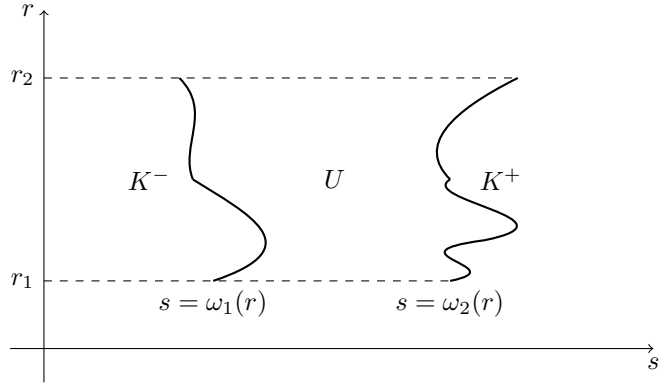
\begin{figure}[ht]
\begin{center}
\begin{tikzpicture}[scale =0.9]
    \draw[->] (-0.5,0) -- (9,0);
    \draw[->] (0,-0.5) -- (0,5);
    \draw[dashed] (0,4)  --  (7,4) ;
    \draw[dashed] (0,1)  --  (6,1) ;
    \draw[thick]   (2, 4) .. controls (2.5,3.5) and  (2,3)   ..(2.2,2.5);
    \draw[thick]   (2.2,2.5) .. controls (3,2) and  (4,1.5)   ..(2.5,1);
    \draw[thick]   (7, 4) .. controls  (6, 3.5) and (5.5,3) ..(6, 2.5 );
    \draw[thick]   (6, 2.5) .. controls  (5.5,2.2) and (8,1.9) ..(6.5,1.6);
    \draw[thick]   (6.5,1.6) .. controls  (5,1.4) and (7,1.2) ..(6,1);
    \draw (0,4) node[left] {$r_2$};
    \draw (0,1) node[left] {$r_1$};
    \draw (0,5) node[left] {$r$};
    \draw (9,0) node[below] {$s$};
    \draw (2.5,1) node[below] {$s=\omega_1(r)$};
    \draw (6,1) node[below] {$s=\omega_2(r)$};
    \draw (2,2.5) node[left] {$K^-$};
    \draw (6.3,2.5) node[right] {$K^+$};
    \draw (4,2.5) node[right] {$U$};
    \end{tikzpicture}
\end{center}
\caption{The right wall $K^+$ and left wall $K^-$.}
\label{fig2-1}
\end{figure}

\subsubsection{Two-wall inclusions}

Let
\[
\Omega_{t_1}^{t_2}=\Omega\times(t_1,t_2)=(0,L)\times(t_1,t_2)\subset\R^2,
\]
where $t_1<t_2$ are any two fixed real numbers, and let $Q\subset \Omega_{t_1}^{t_2}$ be a nonempty open set.
Consider the inhomogeneous partial differential inclusion,
\begin{equation}\label{inclusion}
\nabla z\in K(v)\;\;\mbox{in $Q$},
\end{equation}
where $z=(v,w):Q\to \R^2$.
Regarding this, we fix some terminologies.

\begin{defn}
Let $z=(v,w)\in W^{1,\infty}(Q;\R^2).$ Then the function $z$ is called a \emph{solution} of inclusion (\ref{inclusion}) if
\[
\nabla z\in K(v)\;\;\mbox{a.e. in $Q$},
\]
a \emph{subsolution} of (\ref{inclusion}) if
\[
\nabla z\in K(v)\cup U(v)\;\;\mbox{a.e. in $Q$},
\]
and a \emph{strict subsolution} of (\ref{inclusion}) if
\[
\nabla z\in U(v)\;\;\mbox{a.e. in $Q$},
\]
respectively.
\end{defn}

Observe that if $z=(v,w)\in W^{1,\infty}(Q;\R^2)$ is a solution of (\ref{inclusion}), then
\[
(v_x,w_t)\in K=K^+\cup K^- \;\;\mbox{a.e. in $Q$;}
\]
that is, $(v_x,w_t)$ lies either in the ``right wall'' $K^+$ or in the ``left wall'' $K^-$ almost everywhere in $Q$ (see Figure \ref{fig2-1}).

\subsection{Special solutions to generic problem}
Continuing the previous setup, we present an important existence result on inclusion (\ref{inclusion}) that will serve as the main ingredient for proving Theorems \ref{thm:(ii)-1}, \ref{thm:(ii)-2}, \ref{thm:(iii)-1}, and \ref{thm:(iv)}.

Assume that $z^\star=(v^\star,w^\star)\in C^1(\bar{\Omega}_{t_1}^{t_2};\R^2)$ is a function such that in $Q,$
\begin{equation}\label{subsolution-1}
w^\star_x = v^\star,\;\; r_1<w^\star_t<r_2, \;\;\mbox{and}\;\; \omega_1(w^\star_t)<v^\star_x<\omega_2(w^\star_t).
\end{equation}
From the definition of $U(b)$ $(b\in\R)$,
\[
\nabla z^\star= \begin{pmatrix}
v^\star_x & v^\star_t \\
w^\star_x & w^\star_t
\end{pmatrix} \in U(v^\star)\;\;\mbox{in $Q$;}
\]
that is, $z^\star$ is a strict subsolution of inclusion (\ref{inclusion}).
In particular, we have
\begin{equation}\label{subsolution-3}
(v^\star_x,w^\star_t)\in U \;\;\mbox{in $Q$.}
\end{equation}
Assume further that
\begin{equation}\label{subsolution-2}
(v^\star_x,w^\star_t)(Q) \cap \{(s,r)\in U\,|\, \dist((s,t),\partial U)<\delta\}\ne\emptyset
\end{equation}
for all sufficiently small $\delta>0.$

We are now ready to state the the main result of this section whose proof can be found in \cite{CKK}.

\begin{thm}\label{thm:two-wall}
Let $\epsilon>0.$ Then there exists a function $z=(v,w)\in W^{1,\infty}(\Omega_{t_1}^{t_2};\R^2)$ satisfying the following:
\begin{itemize}
\item[(i)] $z$ is a solution of inclusion (\ref{inclusion}),
\item[(ii)] $z=z^\star$\;\;on $\bar{\Omega}_{t_1}^{t_2}\setminus Q,$
\item[(iii)]$\nabla z=\nabla z^\star$\;\;a.e. on $\Omega_{t_1}^{t_2}\cap\partial Q,$
\item[(iv)] $\|z-z^\star\|_{L^\infty(\Omega_{t_1}^{t_2};\R^2)}<\epsilon,$
\item[(v)] $\|v_t-v^\star_t\|_{L^\infty(\Omega_{t_1}^{t_2})}<\epsilon,$
\item[(vi)] for any nonempty open set $O\subset Q,$
\[
\underset{O}{\mathrm{ess\,osc}}\, v_x \ge d_0,
\]
where $d_0:=\min_{[r_1,r_2]}\omega_2-\max_{[r_1,r_2]}\omega_1>0.$
\end{itemize}
\end{thm}

\section{Proof of main results}\label{sec:proof_main}
This final section is entirely dedicated to proving the main results: Theorems \ref{thm:(ii)-1}, \ref{thm:(ii)-2}, \ref{thm:(iii)-1} and \ref{thm:(iv)}, as well as the classical case, Theorem \ref{thm:(i)}.

\begin{proof}[Proof of Theorem \ref{thm:(i)}]
First, consider the case that $M_0<s^-_0.$ Take $M_1:=\frac{M_0+s^-_0}{2}$.
Using elementary calculus, we can choose a function $\rho^\star\in C^3(\R)$ such that
\[
\rho^\star=\rho\;\;\mbox{on $[-1,M_1]$}\;\;\mbox{and}\;\;\theta_0\le\sigma^\star:=(\rho^\star)'\le \theta_1\;\;\mbox{in $\R$,}
\]
for some constants $\theta_1>\theta_0>0.$
Then from \cite[Theorem 13.24]{Ln}, the modified problem,
\begin{equation*}
\begin{cases}
u^\star_{t} =(\rho^\star(u^\star))_{xx} & \mbox{in $\Omega_\infty$,}\\
u^\star =u_0 & \mbox{on $\Omega\times \{t=0\}$},\\
\sigma^\star(u^\star)u^\star_x=0 & \mbox{on $\partial\Omega\times(0,\infty)$,}
\end{cases}
\end{equation*}
admits a global classical solution $u^\star\in C^{2,1}(\bar\Omega_\infty)$ with
\[
u^\star\in C^{2+a,1+\frac{a}{2}}(\bar\Omega_T)\;\;\forall T>0.
\]
Also, it follows from  Lemma \ref{lem:classical} that
\begin{itemize}
\item[(i)] $
0\le\min_{\bar\Omega}u^\star(\cdot,s)\le u^\star(x,t)\le\max_{\bar\Omega}u^\star(\cdot,s)\le 1\;\;\forall x\in\bar\Omega,\,\forall t\ge s\ge0;
$
\item[(ii)] $
\int_\Omega u^\star(x,t)\,dx=L \bar{u}_0\;\;\mbox{for all $t\ge0$};
$
\item[(iii)] $
\|u^\star(\cdot,t)-\bar{u}_0\|_{L^\infty(\Omega)}\to 0\;\;\mbox{as $t\to\infty$}.
$
\end{itemize}
In particular, $-1<0\le m_0\le u^\star\le M_0<M_1$ in $\Omega_\infty;$
thus, from the choice of $\rho^\star,$ we see that $u:=u^\star$ is a global classical solution to problem (\ref{ib-P}) satisfying (a), (b), and (c).

Next, in the case that $m_0>s^+_0,$ take $m_1:=\frac{s^+_0+m_0}{2}$ and choose a function $\rho^\star\in C^3(\R)$ such that
\[
\rho^\star=\rho\;\;\mbox{on $[m_1,2]$}\;\;\mbox{and}\;\;\theta_2\le\sigma^\star\le \theta_3\;\;\mbox{in $\R$,}
\]
for some constants $\theta_3>\theta_2>0.$ Then the rest of the proof is similar to the above. We omit the details.
\end{proof}

\begin{proof}[Proof of Theorem \ref{thm:(ii)-1}]
In this case, we assume that
\begin{equation}\label{proof-3.2-1}
\bar{u}_0<s^-_2\le s^-_0\le M_0
\end{equation}
and that
\begin{equation}\label{proof-3.2-2}
\max\{\rho(\bar{u}_0),\rho(s^-_1)\}<r_1<r_2<\rho(s^-_2)=r^*.
\end{equation}

In order to fit into the setup in section \ref{sec:generic}, for any $r_1\le r\le r_2$, define
\[
\omega_1(r)=s^-(r)\;\;\mbox{and}\;\;\omega_2(r)=s^+(r);
\]
then $\omega_1,\omega_2\in C([r_1,r_2])$, and
\[
\max_{[r_1,r_2]}\omega_1=s^-(r_2)<s^+(r_1)=\min_{[r_1,r_2]}\omega_2.
\]

Next, using elementary calculus, we can choose a function $\rho^\star\in C^3(\R)$ such that
\begin{equation}\label{rho-modi}
\begin{cases}
  \rho^\star=\rho\;\;\mbox{on $[-1,s^-(r_1)]\cup[s^+(r_2),2]$,} \\[1mm]
  \mbox{$\exists\theta_1>\theta_0>0$ s.t.}\;\theta_0\le\sigma^\star:=(\rho^\star)'\le\theta_1\;\;\mbox{in $\R,$} \\[1mm]
  \rho^\star<\rho\;\;\mbox{on $(s^-(r_1),s^-(r_2)],$}\;\;\mbox{and}\;\;\rho^\star>\rho\;\;\mbox{on $[s^+(r_1),s^+(r_2)).$}
\end{cases}
\end{equation}
Then from \cite[Theorem 13.24]{Ln} and Lemma \ref{lem:classical}, the modified problem,
\begin{equation}\label{ib-modi}
\begin{cases}
u^\star_{t} =(\sigma^\star(u^\star)u^\star_x)_x=(\rho^\star(u^\star))_{xx} & \mbox{in $\Omega_\infty$,}\\
u^\star =u_0 & \mbox{on $\Omega\times \{t=0\}$},\\
\sigma^\star(u^\star)u^\star_x=0 & \mbox{on $\partial\Omega\times(0,\infty)$,}
\end{cases}
\end{equation}
possesses a global solution \underline{$u^\star\in C^{2,1}(\bar\Omega_\infty;[0,1])$ with $u^\star\in C^{2+a,1+\frac{a}{2}}(\bar\Omega_T;[0,1])$ for} \underline{each $T>0$} such that
\begin{itemize}
\item[(i)] $
0\le\min_{\bar\Omega}u^\star(\cdot,s)\le u^\star(x,t)\le\max_{\bar\Omega}u^\star(\cdot,s)\le 1\;\;\forall x\in\bar\Omega,\,\forall t\ge s\ge0;
$
\item[(ii)] $
\int_\Omega u^\star(x,t)\,dx=L \bar{u}_0\;\;\mbox{for all $t\ge0$};
$
\item[(iii)] $
\|u^\star(\cdot,t)-\bar{u}_0\|_{L^\infty(\Omega)}\to 0\;\;\mbox{as $t\to\infty$}.
$
\end{itemize}
Thus, \underline{(e) and (f) in Theorem \ref{thm:(ii)-1} are satisfied.}

We define
\begin{equation}\label{def:vstar}
v^\star(x,t)=\int_0^x u^\star(y,t)\,dy+ \int_0^t u^\star_x(0,s)\,ds\;\;\forall (x,t)\in\bar\Omega_\infty;
\end{equation}
then from (\ref{rho-modi}) and (\ref{ib-modi}), $v^\star\in C^{3,1}(\bar\Omega_\infty)$ satisfies that for all $(x,t)\in\Omega_\infty$,
\[
\begin{split}
v^\star_t(x,t) & = \int_0^x u^\star_t(y,t)\,dy+ u^\star_x(0,t) = \int_0^x (\sigma^\star(u^\star)u^\star_x)_x(y,t)\,dy+ u^\star_x(0,t)\\
& = \sigma^\star(u^\star(x,t))u^\star_x(x,t) - \sigma^\star(u^\star(0,t))u^\star_x(0,t) + u^\star_x(0,t)\\
& = \sigma^\star(v^\star_x(x,t))v^\star_{xx}(x,t) = (\rho^\star(v^\star_x))_{x}(x,t).
\end{split}
\]
Hence, $v^\star$ is a global solution to the problem,
\begin{equation}\label{ib-modi-1}
\begin{cases}
v^\star_{t} =(\rho^\star(v^\star_x))_{x} & \mbox{in $\Omega_\infty$,}\\
v^\star =v_0 & \mbox{on $\Omega\times \{t=0\}$},\\
\sigma^\star(v^\star_x)v^\star_{xx}=0 & \mbox{on $\partial\Omega\times(0,\infty)$,}
\end{cases}
\end{equation}
where
\[
v_0(x):=\int_0^x u_0(y)\,dy \;\;\forall x\in\bar\Omega.
\]

In turn, we define
\[
w^\star(x,t)=\int_0^t \rho^\star (v^\star_x(x,s))\,ds + \int_0^x v_0(y)\,dy\;\;\forall (x,t)\in\bar\Omega_\infty;
\]
then from (\ref{def:vstar}) and (\ref{ib-modi-1}),
\begin{equation}\label{system-1}
\begin{cases}
w^\star_{x} =v^\star \\
w^\star_t =\rho^\star(v^\star_x)
\end{cases}
\;\;\mbox{in $\Omega_\infty$},
\end{equation}
$v^\star_x=u^\star\in C^{2,1}(\bar\Omega_\infty;[0,1])$, and $z^\star:=(v^\star,w^\star)\in (C^{3,1}\times C^{4,2})(\bar\Omega_\infty;\R^2).$

Define
\begin{equation}\label{proof-3.2-3}
Q=\{(x,t)\in\Omega_\infty\,|\, s^-(r_1)<v^\star_x(x,t)<s^+(r_2)\}.
\end{equation}
From (\ref{proof-3.2-1}) and (\ref{proof-3.2-2}), we have $\bar{u}_0<s^-(r_1)$. So it follows from (iii) above that there exists a finite time $T_1>0$ with
\[
v^\star_x(x,t)=u^\star(x,t)< s^-(r_1)\;\;\forall (x,t)\in\Omega\times[T_1,\infty);
\]
thus,
\[
Q\subset \Omega\times(0,T_1).
\]
Since $\bar{u}_0<s^-(r_1)<s^-_2\le s^-_0\le M_0$, we can take a point $x_0\in\Omega$ such that
\[
s^-(r_1)<v^\star_x(x_0,0)=u_0(x_0)<s^+(r_2).
\]
Then by continuity, we can take an $r_0\in(0,T_1)$ with $r_0<\min\{x_0,L-x_0\}$ so small that
\[
s^-(r_1)<v^\star_x(x,t)<s^+(r_2)\;\;\forall(x,t)\in (\Omega\times[0,\infty)) \cap B_{r_0}(x_0,0);
\]
thus, $\Omega_\infty \cap B_{r_0}(x_0,0)\subset Q\ne\emptyset$ so that
\[
(x_0-r_0,x_0+r_0)\times\{0\}\subset\underline{\bar{Q}\cap(\Omega\times\{0\})\ne\emptyset}
\]
and that \underline{$Q$ is a nonempty bounded open subset of $\Omega_\infty.$}

Following the notations in section \ref{sec:generic}, note from (\ref{proof-3.2-1}), (\ref{proof-3.2-2}), (\ref{rho-modi}), (\ref{system-1}), and (\ref{proof-3.2-3}) that (\ref{subsolution-1}) holds in $Q$ and (\ref{subsolution-2}) holds for all sufficiently small $\delta>0$.
Thus, for any fixed $\epsilon>0,$ we can apply Theorem \ref{thm:two-wall} (with $t_1=0$ and $t_2=T_1$ in the current case) to obtain a function $z=z_\epsilon\in W^{1,\infty}(\Omega^{T_1}_0;\R^2)$ satisfying the following, where $z=(v,w)$:
\begin{itemize}
\item[(1)] $z$ is a solution of inclusion (\ref{inclusion});
\item[(2)] $z=z^\star$\;\;on $\bar{\Omega}^{T_1}_0\setminus Q;$
\item[(3)]$\nabla z=\nabla z^\star$\;\;a.e. on $\Omega^{T_1}_0\cap\partial Q;$
\item[(4)] $\|z-z^\star\|_{L^\infty(\Omega^{T_1}_0;\R^2)}<\epsilon;$
\item[(5)] $\|v_t-v^\star_t\|_{L^\infty(\Omega^{T_1}_0)}<\epsilon;$
\item[(6)] for any nonempty open set $O\subset Q,$
\[
\underset{O}{\mathrm{ess\,osc}}\, v_x \ge d_0,
\]
where $d_0=\min_{[r_1,r_2]}\omega_2-\max_{[r_1,r_2]}\omega_1=s^+(r_1)-s^-(r_2)>0.$
\end{itemize}
For each $(x,t)\in\bar{\Omega}\times [T_1,\infty)$, define
\begin{equation}\label{def:z-1}
z(x,t)=z^\star(x,t);
\end{equation}
then from (2), $z=(v,w)\in W^{1,\infty}(\Omega_T;\R^2)$ for all $T>0$.

For each $t\ge 0,$ define
\begin{equation}\label{def:u-sol}
u(\cdot,t)=v_x(\cdot,t)\;\;\mbox{a.e. in $\Omega$};
\end{equation}
then $u\in L^\infty(\Omega_T)$ for all $T>0$.
First, from (6), \underline{(c) in Theorem \ref{thm:(ii)-1} is fulfilled.}
Also, note from (2) and (\ref{def:vstar}) that for every $t\ge 0,$
\[
\begin{split}
\int_\Omega u(x,t)\,dx & =\int_\Omega v_x(x,t)\,dx=v(L,t)-v(0,t) \\
&= v^\star(L,t)-v^\star(0,t) =\int_\Omega v^\star_x(x,t)\,dx= \int_\Omega u^\star(x,t)\,dx;
\end{split}
\]
hence, \underline{(d) in Theorem \ref{thm:(ii)-1} follows from (ii).}
From (2), (\ref{def:vstar}), and (\ref{def:z-1}), we have
\[
u=v_x=v^\star_x=u^\star\in[0,1]\;\;\mbox{in $\Omega_\infty\setminus\bar{Q}$},
\]
and from (3) and (\ref{def:vstar}), we get
\[
u=v_x=v^\star_x=u^\star\in[0,1]\;\;\mbox{a.e. in $\Omega_\infty\cap\partial Q$}.
\]
In particular, \underline{(a) in Theorem \ref{thm:(ii)-1} is satisfied.}
From (1), we have
\[
\nabla z\in K(v)\;\;\mbox{a.e. in $Q$;}
\]
that is, a.e. in $Q$,
\begin{equation}\label{property-v-w}
\begin{cases}
  u=v_x\in[s^-(r_1),s^-(r_2)]\cup[s^+(r_1),s^+(r_2)],  \\
  w_t=\rho(v_x), \\
  w_x=v.
\end{cases}
\end{equation}
In particular, we see that $u\in[0,1]$ a.e. in $\Omega_\infty$, that is, \underline{$u\in L^\infty(\Omega_\infty;[0,1])$} and that \underline{(b) in  Theorem \ref{thm:(ii)-1} holds.}

Now, we check that $u$ is a global weak solution to problem (\ref{ib-P}). To do so, fix any $T>0$ and test function $\varphi\in C^\infty(\bar\Omega\times[0,T])$ with
\begin{equation}\label{bdry:testfn}
\varphi=0\;\;\mbox{on $\bar\Omega\times\{t=T\}$}\;\;\mbox{and}\;\;\varphi_x=0\;\;\mbox{on $\partial\Omega\times[0,T]$}.
\end{equation}
Observe from (2), (3), (\ref{rho-modi}), (\ref{def:vstar}), (\ref{ib-modi-1}), (\ref{system-1}), (\ref{proof-3.2-3}), (\ref{def:u-sol}), (\ref{property-v-w}), (\ref{bdry:testfn}), and the integration by parts that
\[
\begin{split}
\int_0^T & \int_0^L(u\varphi_t+\rho(u)\varphi_{xx})\,dxdt = \int_0^T\int_0^L (v_x\varphi_t+w_t\varphi_{xx})\,dxdt \\
= & -\int_0^T\int_0^L (v\varphi_{tx}+w\varphi_{xxt})\,dxdt +\int_0^T (v(L,t)\varphi_t(L,t)-v(0,t)\varphi_t(0,t))\,dt \\
& +\int_0^L (w(x,T)\varphi_{xx}(x,T)-w(x,0)\varphi_{xx}(x,0)) \,dx  \\
= & -\int_0^T\int_0^L (v\varphi_{tx}-w_x\varphi_{xt})\,dxdt +\int_0^T (v(L,t)\varphi_t(L,t)-v(0,t)\varphi_t(0,t))\,dt \\
& +\int_0^L (w(x,T)\varphi_{xx}(x,T)-w(x,0)\varphi_{xx}(x,0)) \,dx  \\
& +\int_0^T (w(0,t)\varphi_{xt}(0,t)-w(L,t)\varphi_{xt}(L,t)) \,dt \\
= & \int_0^T (v^\star(L,t)\varphi_t(L,t)-v^\star(0,t)\varphi_t(0,t))\,dt - \int_0^L w^\star(x,0)\varphi_{xx}(x,0) \,dx  \\
= & -\int_0^T (v^\star_t(L,t)\varphi(L,t)-v^\star_t(0,t)\varphi(0,t))\,dt - \int_0^L w^\star(x,0)\varphi_{xx}(x,0) \,dx  \\
& +v^\star(L,T)\varphi(L,T)-v^\star(L,0)\varphi(L,0)-v^\star(0,T)\varphi(0,T) +v^\star(0,0)\varphi(0,0) \\
= & - \int_0^L w^\star(x,0)\varphi_{xx}(x,0) \,dx  -v^\star(L,0)\varphi(L,0)\\
= & \int_0^L v^\star(x,0)\varphi_{x}(x,0)\,dx -v^\star(L,0)\varphi(L,0) +w^\star(0,0)\varphi_{x}(0,0)-w^\star(L,0)\varphi_{x}(L,0)\\
= & -\int_0^L u_0(x)\varphi(x,0)\,dx.
\end{split}
\]
Thus, according to Definition \ref{def:global-weak-sol}, \underline{$u$ is a global weak solution to (\ref{ib-P}).}

Since $u^\star$ itself is not a global weak solution to problem (\ref{ib-P}), it follows from (4) that \underline{there are infinitely many global weak solutions to (\ref{ib-P}) that satisfy properties} \underline{(a)--(f) in Theorem \ref{thm:(ii)-1}.}

The proof of Theorem \ref{thm:(ii)-1} is now complete.
\end{proof}

\begin{proof}[Proof of Theorem \ref{thm:(ii)-2}]
In this case, we assume that problem (\ref{ib-P}) is of type (II),
\begin{equation}\label{proof-3.3-0-0}
s^-_2\le\bar{u}_0<s^-_0\le M_0,
\end{equation}
and
\begin{equation}\label{proof-3.3-0-1}
\rho(s^-_1)\le r_1<r_2\le r^*=\rho(1).
\end{equation}

In order to fit into the setup in section \ref{sec:generic}, for any $r_1\le r\le r_2$, define
\[
\omega_1(r)=s^-(r)\;\;\mbox{and}\;\;\omega_2(r)=s^+(r);
\]
then $\omega_1,\omega_2\in C([r_1,r_2])$, and
\[
\max_{[r_1,r_2]}\omega_1=s^-(r_2)<s^+(r_1)=\min_{[r_1,r_2]}\omega_2.
\]

Next, using elementary calculus, we can choose a function $\rho^\star_1\in C^3(\R)$ such that
\begin{equation}\label{proof-3.3-1}
\begin{cases}
  \rho^\star_1=\rho\;\;\mbox{on $[-1,s^-(r_1)]\cup[s^+(r_2),2]$,} \\[1mm]
  \mbox{$\exists\theta_1>\theta_0>0$ s.t.}\;\theta_0\le\sigma^\star_1:=(\rho^\star_1)'\le\theta_1\;\;\mbox{in $\R,$} \\[1mm]
  \rho^\star_1<\rho\;\;\mbox{on $(s^-(r_1),s^-(r_2)],$}\;\;\mbox{and}\;\;\rho^\star_1>\rho\;\;\mbox{on $[s^+(r_1),s^+(r_2)).$}
\end{cases}
\end{equation}
Then from \cite[Theorem 13.24]{Ln} and Lemma \ref{lem:classical}, the modified problem,
\begin{equation}\label{proof-3.3-2}
\begin{cases}
u^\star_{t} =(\sigma^\star_1(u^\star)u^\star_x)_x=(\rho^\star_1(u^\star))_{xx} & \mbox{in $\Omega_\infty$,}\\
u^\star =u_0 & \mbox{on $\Omega\times \{t=0\}$},\\
\sigma^\star_1(u^\star)u^\star_x=0 & \mbox{on $\partial\Omega\times(0,\infty)$,}
\end{cases}
\end{equation}
possesses a global solution \underline{$u^\star_1\in C^{2,1}(\bar\Omega_\infty;[0,1])$ with $u^\star_1\in C^{2+a,1+\frac{a}{2}}(\bar\Omega_T;[0,1])$ for} \underline{each $T>0$} such that
\begin{itemize}
\item[(i$_1$)] $
0\le\min_{\bar\Omega}u^\star_1(\cdot,s)\le u^\star_1(x,t)\le\max_{\bar\Omega}u^\star_1(\cdot,s)\le 1\;\;\forall x\in\bar\Omega,\,\forall t\ge s\ge0;
$
\item[(ii$_1$)] $
\int_\Omega u^\star_1(x,t)\,dx=L \bar{u}_0\;\;\mbox{for all $t\ge0$};
$
\item[(iii$_1$)] $
\|u^\star_1(\cdot,t)-\bar{u}_0\|_{L^\infty(\Omega)}\to 0\;\;\mbox{as $t\to\infty$}.
$
\end{itemize}

Thanks to (iii$_1$), we can take a finite time $T_1>0$ such that
\begin{equation}\label{proof-3.3-3}
u^\star_1(x,t)\le \frac{\bar{u}_0+s^-_0}{2}\;\;\forall x\in\bar\Omega,\,\forall t\ge T_1.
\end{equation}
Define
\begin{equation}\label{proof-3.3-3-0}
u_1(x)=u^\star_1(x,T_1)\;\;\forall x\in\bar\Omega;
\end{equation}
then
\[
u_1\in C^{2+a}(\bar\Omega;[0,1])\;\;\mbox{and}\;\;u_1'(L)=u_1'(0)=0.
\]
Also, from (i$_1$) and (\ref{proof-3.3-3}), we have
\begin{equation}\label{proof-3.3-3-1}
m_0\le m_1:=\min_{\bar\Omega}u_1\le\max_{\bar\Omega}u_1=: M_1\le \frac{\bar{u}_0+s^-_0}{2}<s^-_0\le M_0.
\end{equation}
Relying on elementary calculus, we can choose a function $\rho^\star_2\in C^3(\R)$ such that
\begin{equation}\label{proof-3.3-3-2}
\rho^\star_2=\rho\;\;\mbox{on $[-1,M_1]$}\;\;\mbox{and}\;\;\theta_2\le\sigma^\star_2 :=(\rho^\star_2)'\le\theta_3\;\;\mbox{in $\R$},
\end{equation}
for some constants $\theta_3>\theta_2>0.$
Then from \cite[Theorem 13.24]{Ln} and Lemma \ref{lem:classical}, the modified problem,
\begin{equation}\label{proof-3.3-4}
\begin{cases}
u^\star_{t} =(\rho^\star_2(u^\star))_{xx} & \mbox{in $\Omega\times(T_1,\infty)$,}\\
u^\star =u_1 & \mbox{on $\Omega\times \{t=T_1\}$},\\
\sigma^\star_2(u^\star)u^\star_x=0 & \mbox{on $\partial\Omega\times(T_1,\infty)$,}
\end{cases}
\end{equation}
admits a global classical solution \underline{$u^\star_2\in C^{2,1}(\bar\Omega\times[T_1,\infty);[0,1])$ with}
\[
\underline{u^\star_2\in C^{2+a,1+\frac{a}{2}}(\bar\Omega\times[T_1,T];[0,1])\;\;\forall T>T_1}
\]
satisfying
\begin{itemize}
\item[(i$_2$)] $
0\le\min_{\bar\Omega}u^\star_2(\cdot,s)\le u^\star_2(x,t)\le\max_{\bar\Omega}u^\star_2(\cdot,s)\le 1\;\;\forall x\in\bar\Omega,\,\forall t\ge s\ge T_1;
$
\item[(ii$_2$)] $
\int_\Omega u^\star_2(x,t)\,dx=L \bar{u}_1\;\;\mbox{for all $t\ge T_1$};
$
\item[(iii$_2$)] $
\|u^\star_2(\cdot,t)-\bar{u}_1\|_{L^\infty(\Omega)}\to 0\;\;\mbox{as $t\to\infty$}.
$
\end{itemize}
Note from (\ref{proof-3.3-3-1}), (\ref{proof-3.3-3-2}), and (i$_2$) that $u^\star_2$ is a global classical solution to problem (\ref{proof-3.3-4}) with $\rho^\star_2$ and $\sigma^\star_2$ replaced by $\rho$ and $\sigma$, respectively.
In particular, \underline{(\ref{proof-3.3-3-1}) and (i$_2$)} \underline{imply (e) in Theorem \ref{thm:(ii)-2}.}
Also, note from (ii$_1$), (\ref{proof-3.3-4}), and the definition of $u_1$ that $\bar{u}_1=\bar{u}_0$ and that \underline{$u^\star_2=u^\star_1$ on $\bar\Omega\times\{t=T_1\};$} thus, \underline{(f) in Theorem \ref{thm:(ii)-2} follows from} \underline{(iii$_2$).}

For every $(x,t)\in\bar\Omega_{T_1}$, we define
\begin{equation}\label{proof-3.3-4-1}
v^\star(x,t)=\int_0^x u_1^\star(y,t)\,dy+\int_0^t (u_1^\star)_x(0,s)\,ds
\end{equation}
and
\[
w^\star(x,t)=\int_0^t \rho_1^\star(v^\star_x(x,s))\,ds+\int_0^x v_0(y)\,dy,
\]
where
\[
v_0(x):=\int_0^x u_0(y)\,dy\;\;\forall x\in\bar\Omega.
\]
Then $z^\star:=(v^\star,w^\star)\in(C^{3,1}\times C^{4,2})(\bar\Omega_{T_1})$ satisfies
\begin{equation}\label{proof-3.3-5}
\begin{cases}
  v^\star_t=(\rho^\star_1(v^\star_x))_x & \mbox{in $\Omega_{T_1}$,} \\
  w^\star_x=v^\star & \mbox{in $\Omega_{T_1}$,} \\
  w^\star_t=\rho^\star_1(v^\star_x) & \mbox{in $\Omega_{T_1}$,} \\
  v^\star=v_0 & \mbox{on $\Omega\times\{t=0\}$,} \\
  \sigma^\star_1(v^\star_x)v^\star_{xx}=0 & \mbox{on $\partial\Omega\times(0,T_1)$.}
\end{cases}
\end{equation}

Define
\begin{equation}\label{proof-3.3-6}
Q=\{(x,t)\in\Omega_{T_1}\,|\, s^-(r_1)<v^\star_x(x,t)<s^+(r_2)\}.
\end{equation}
Thanks to (\ref{proof-3.3-0-0}) and (\ref{proof-3.3-0-1}), we can pick a point $x_0\in\Omega$ such that
\[
s^-(r_1)<v^\star_x(x_0,0)=u_0(x_0)<s^+(r_2).
\]
So we can deduce that \underline{$Q$ is a nonempty bounded open subset of $\Omega_{T_1}$ with}
\[
\underline{\bar{Q}\cap(\Omega\times\{0\})\ne\emptyset.}
\]

Following the notations in section \ref{sec:generic}, note from (\ref{proof-3.3-1}), (\ref{proof-3.3-5}), and (\ref{proof-3.3-6}) that (\ref{subsolution-1}) holds in $Q$.
Thus, for any fixed $\epsilon>0,$ we can apply Theorem \ref{thm:two-wall} (with $t_1=0$ and $t_2=T_1$ in the current case) to obtain a function $z=z_\epsilon\in W^{1,\infty}(\Omega^{T_1}_0;\R^2)$ satisfying the following, where $z=(v,w)$:
\begin{itemize}
\item[(1)] $z$ is a solution of inclusion (\ref{inclusion});
\item[(2)] $z=z^\star$\;\;on $\bar{\Omega}^{T_1}_0\setminus Q;$
\item[(3)]$\nabla z=\nabla z^\star$\;\;a.e. on $\Omega^{T_1}_0\cap\partial Q;$
\item[(4)] $\|z-z^\star\|_{L^\infty(\Omega^{T_1}_0;\R^2)}<\epsilon;$
\item[(5)] $\|v_t-v^\star_t\|_{L^\infty(\Omega^{T_1}_0)}<\epsilon;$
\item[(6)] for any nonempty open set $O\subset Q,$
\[
\underset{O}{\mathrm{ess\,osc}}\, v_x \ge d_0,
\]
where $d_0=\min_{[r_1,r_2]}\omega_2-\max_{[r_1,r_2]}\omega_1=s^+(r_1)-s^-(r_2)>0.$
\end{itemize}

For each $0\le t< T_1,$ define
\begin{equation}\label{proof-3.3-6-1}
u(\cdot,t)=v_x(\cdot,t)\;\;\mbox{a.e. in $\Omega$},
\end{equation}
and for each $(x,t)\in\bar{\Omega}\times [T_1,\infty)$, define
\begin{equation}\label{proof-3.3-7}
u(x,t)=u^\star_2(x,t);
\end{equation}
then $u\in L^\infty(\Omega_\infty).$ First, note from (6) that \underline{(c) in Theorem \ref{thm:(ii)-2} holds.}
Also, it follows from (2) and (\ref{proof-3.3-4-1}) that for $0\le t<T_1,$
\[
\begin{split}
\int_\Omega u(x,t)\,dx= & \,\int_\Omega v_x(x,t)\,dx=v(L,t)-v(0,t)\\
= &\, v^\star(L,t)-v^\star(0,t)= \int_\Omega v^\star_x(x,t)\,dx= \int_\Omega u_1^\star(x,t)\,dx.
\end{split}
\]
This together with (ii$_1$) and (ii$_2$) implies \underline{(d) in Theorem \ref{thm:(ii)-2}.}
From (2) and (\ref{proof-3.3-4-1}),
\[
u=v_x=v^\star_x=u^\star_1\in[0,1]\;\;\mbox{in $\Omega_{T_1}\setminus \bar{Q}$},
\]
and from (3) and (\ref{proof-3.3-4-1}),
\[
u=v_x=v^\star_x=u^\star_1\in[0,1]\;\;\mbox{a.e. in $\Omega_{T_1}\cap\partial Q$}.
\]
In particular, with (\ref{proof-3.3-7}), \underline{(a) in Theorem \ref{thm:(ii)-2} is satisfied.} From (1), a.e. in $Q$,
\begin{equation}\label{proof-3.3-7-1}
\begin{cases}
  u=v_x\in[s^-(r_1),s^-(r_2)]\cup[s^+(r_1),s^+(r_2)],  \\
  w_t=\rho(v_x), \\
  w_x=v.
\end{cases}
\end{equation}
We can now conclude that \underline{$u\in L^\infty(\Omega_\infty;[0,1])$} and that \underline{(b) in Theorem \ref{thm:(ii)-2} holds.}

Next, we show that $u$ is a global weak solution to problem (\ref{ib-P}). Choose any $T>0$ and test function $\varphi\in C^\infty(\bar\Omega\times[0,T])$ with
\begin{equation}\label{proof-3.3-8}
\varphi=0\;\;\mbox{on $\bar\Omega\times\{t=T\}$}\;\;\mbox{and}\;\;\varphi_x=0\;\;\mbox{on $\partial\Omega\times[0,T]$}.
\end{equation}
We have to check that equality (\ref{eq:weaksol}) holds. Since the case that $T\le T_1$ can be handled similarly as in the proof of Theorem \ref{thm:(ii)-1}, we only deal with the case that $T>T_1.$ First, we decompose
\[
\begin{split}
\int_0^T\int_0^L(u\varphi_t &+\rho(u)\varphi_{xx})\,dxdt \\
= & \bigg(\int_0^{T_1}\int_0^L+ \int_{T_1}^T\int_0^L\bigg) (u\varphi_t+\rho(u)\varphi_{xx})\,dxdt=:I_1+I_2.
\end{split}
\]
From Proposition \ref{prop:classical-global} and (\ref{proof-3.3-4}) with its following remark, we have
\[
I_2=-\int_{0}^L u_1(x)\varphi(x,T_1)\,dx.
\]
Observe from (2), (3), (\ref{proof-3.3-1}), (\ref{proof-3.3-2}), (\ref{proof-3.3-3-0}), (\ref{proof-3.3-4-1}), (\ref{proof-3.3-5}), (\ref{proof-3.3-6}), (\ref{proof-3.3-6-1}), (\ref{proof-3.3-7-1}),  and  (\ref{proof-3.3-8})   that
\[
\begin{split}
I_1 = & \int_0^{T_1}\int_0^L (v_x\varphi_t+w_t\varphi_{xx})\,dxdt \\
= & -\int_0^{T_1}\int_0^L (v\varphi_{tx}+w\varphi_{xxt})\,dxdt +\int_0^{T_1} (v(L,t)\varphi_t(L,t)-v(0,t)\varphi_t(0,t))\,dt \\
& +\int_0^L (w(x,T_1)\varphi_{xx}(x,T_1)-w(x,0)\varphi_{xx}(x,0)) \,dx  \\
= & -\int_0^{T_1}\int_0^L (v\varphi_{tx}-w_x\varphi_{xt})\,dxdt +\int_0^{T_1} (v(L,t)\varphi_t(L,t)-v(0,t)\varphi_t(0,t))\,dt \\
& +\int_0^L (w(x,T_1)\varphi_{xx}(x,T_1)-w(x,0)\varphi_{xx}(x,0)) \,dx  \\
& +\int_0^{T_1} (w(0,t)\varphi_{xt}(0,t)-w(L,t)\varphi_{xt}(L,t)) \,dt \\
= & \int_0^{T_1} (v^\star(L,t)\varphi_t(L,t)-v^\star(0,t)\varphi_t(0,t))\,dt \\
& + \int_0^L (w^\star(x,T_1)\varphi_{xx}(x,T_1)-w^\star(x,0)\varphi_{xx}(x,0)) \,dx  \\
= & -\int_0^{T_1} (v^\star_t(L,t)\varphi(L,t)-v^\star_t(0,t)\varphi(0,t))\,dt \\
& + \int_0^L (w^\star(x,T_1)\varphi_{xx}(x,T_1)-w^\star(x,0)\varphi_{xx}(x,0)) \,dx \\
& +v^\star(L,T_1)\varphi(L,T_1)-v^\star(L,0)\varphi(L,0)-v^\star(0,T_1)\varphi(0,T_1) +v^\star(0,0)\varphi(0,0) \\
= & \int_0^L (w^\star(x,T_1)\varphi_{xx}(x,T_1)-w^\star(x,0)\varphi_{xx}(x,0)) \,dx \\
& +v^\star(L,T_1)\varphi(L,T_1)-v^\star(L,0)\varphi(L,0)-v^\star(0,T_1)\varphi(0,T_1) \\
= & \int_0^L (-v^\star(x,T_1)\varphi_{x}(x,T_1)+v^\star(x,0)\varphi_{x}(x,0))\,dx \\
& +v^\star(L,T_1)\varphi(L,T_1)-v^\star(L,0)\varphi(L,0)-v^\star(0,T_1)\varphi(0,T_1) \\
& +w^\star(L,T_1)\varphi_x(L,T_1)-w^\star(0,T_1)\varphi_x(0,T_1) -w^\star(L,0)\varphi_x(L,0)+w^\star(0,0)\varphi_x(0,0) \\
= & \int_0^L (u^\star_1(x,T_1)\varphi(x,T_1)-u^\star_1(x,0)\varphi(x,0))\,dx \\
& +v^\star(L,T_1)\varphi(L,T_1)-v^\star(L,0)\varphi(L,0)-v^\star(0,T_1)\varphi(0,T_1) \\
& -v^\star(L,T_1)\varphi(L,T_1)+v^\star(0,T_1)\varphi(0,T_1)+v^\star(L,0)\varphi(L,0) -v^\star(0,0)\varphi(0,0) \\
= & \int_0^L (u_1(x)\varphi(x,T_1)-u_0(x,0)\varphi(x,0))\,dx.
\end{split}
\]
Thus, summing $I_1$ and $I_2$, we obtain (\ref{eq:weaksol}); hence, \underline{$u$ is a global weak solution to} \underline{(\ref{ib-P}).}

Since $\chi_{[0,T_1)}(t)u^\star_1(\cdot,t)+\chi_{[T_1,\infty)}(t)u^\star_2(\cdot,t)$ $(t\ge0)$ itself is not a global weak solution to problem (\ref{ib-P}), it follows from (4) that \underline{there are infinitely many global weak} \underline{solutions to (\ref{ib-P}) satisfying properties (a)--(f) in Theorem \ref{thm:(ii)-2}.}

The proof of Theorem \ref{thm:(ii)-2} is now complete.
\end{proof}

\begin{proof}[Proof of Theorem \ref{thm:(iv)}]
In this case, we assume that
\begin{equation}\label{proof-3.5-1}
\bar{u}_0\in[s^-_0,s^+_0]
\end{equation}
and
\begin{equation}\label{proof-3.5-2}
\rho(s^+_0)<r_0<r^*.
\end{equation}
Then we choose a strictly increasing sequence $\{r_{1k}\}_{k\in\N}$ in $(\rho(s^+_0),r_0)$ and a strictly decreasing sequence $\{r_{2k}\}_{k\in\N}$ in $(r_0,r^*)$ such that
\[
\lim_{k\to\infty}r_{1k}=\lim_{k\to\infty}r_{2k}=r_0.
\]

In order to fit into the setup in section \ref{sec:generic}, for any $\rho(s^+_0)\le r\le r^*$, define
\[
\omega_1(r)=s^-(r)\;\;\mbox{and}\;\;\omega_2(r)=s^+(r);
\]
then $\omega_1,\omega_2\in C([\rho(s^+_0),r^*])$, and
\[
\max_{[\rho(s^+_0),r^*]}\omega_1=s^-_2\le s^-_0<s^+_0=\min_{[\rho(s^+_0),r^*]}\omega_2.
\]

Next, using elementary calculus, for each $k\in\N$, we can choose a function $\rho^\star_k\in C^3(\R)$ such that
\begin{equation}\label{proof-3.5-3}
\begin{cases}
  \rho^\star_k=\rho\;\;\mbox{on $[-1,s^-(r_{1k})]\cup[s^+(r_{2k}),2]$,} \\[1mm]
  \mbox{$\exists\theta_{1k}>\theta_{0k}>0$ s.t.}\;\theta_{0k}\le\sigma^\star_k:=(\rho^\star_k)'\le\theta_{1k}\;\;\mbox{in $\R,$} \\[1mm]
  \rho^\star_k<\rho\;\;\mbox{on $(s^-(r_{1k}),s^-(r_{2k})],$}\;\;\mbox{and}\;\;\rho^\star_k>\rho\;\;\mbox{on $[s^+(r_{1k}),s^+(r_{2k})).$}
\end{cases}
\end{equation}

From \cite[Theorem 13.24]{Ln} and Lemma \ref{lem:classical}, the modified problem,
\begin{equation}\label{proof-3.5-4}
\begin{cases}
u^\star_{t} =(\sigma^\star_1(u^\star)u^\star_x)_x=(\rho^\star_1(u^\star))_{xx} & \mbox{in $\Omega_\infty$,}\\
u^\star =u_0 & \mbox{on $\Omega\times \{t=0\}$},\\
\sigma^\star_1(u^\star)u^\star_x=0 & \mbox{on $\partial\Omega\times(0,\infty)$,}
\end{cases}
\end{equation}
possesses a global solution \underline{$u^\star_1\in C^{2,1}(\bar\Omega_\infty;[0,1])$ with $u^\star_1\in C^{2+a,1+\frac{a}{2}}(\bar\Omega_T;[0,1])$ for} \underline{each $T>0$} such that
\begin{itemize}
\item[(i$_1$)] $
0\le\min_{\bar\Omega}u^\star_1(\cdot,s)\le u^\star_1(x,t)\le\max_{\bar\Omega}u^\star_1(\cdot,s)\le 1\;\;\forall x\in\bar\Omega,\,\forall t\ge s\ge0;
$
\item[(ii$_1$)] $
\int_\Omega u^\star_1(x,t)\,dx=L \bar{u}_0\;\;\mbox{for all $t\ge0$};
$
\item[(iii$_1$)] $
\|u^\star_1(\cdot,t)-\bar{u}_0\|_{L^\infty(\Omega)}\to 0\;\;\mbox{as $t\to\infty$}.
$
\end{itemize}
Thanks to (\ref{proof-3.5-1}), (\ref{proof-3.5-2}), and (iii$_1$), we can take a finite time $T_1>0$ such that
\begin{equation}\label{proof-3.5-5}
s^-(r_0)<u^\star_1(x,T_1)<s^+(r_0)\;\;\forall x\in\bar\Omega.
\end{equation}
Define
\[
u_1(x)=u^\star_1(x,T_1)\;\;\forall x\in\bar\Omega;
\]
then $u_1\in C^{2+a}(\bar\Omega;[0,1])$ and $u_1'(0)=u_1'(L)=0.$ Note also from (ii$_1$) that $\bar{u}_1=\bar{u}_0$.
For later use, let us take $T_j:=T_1+j-1$ for all $j\ge 2$ and $T_0:=0$.

Having chosen a function $u_k\in C^{2+a}(\bar\Omega;[0,1])$ with $u_k'(0)=u_k'(L)=0$ and $\bar{u}_k=\bar{u}_{k-1}$ for some $k\in\N,$ it follows from \cite[Theorem 13.24]{Ln} and Lemma \ref{lem:classical} that the modified problem,
\begin{equation}\label{proof-3.5-6}
\begin{cases}
u^\star_{t} =(\sigma^\star_{k+1}(u^\star)u^\star_x)_x=(\rho^\star_{k+1}(u^\star))_{xx} & \mbox{in $\Omega\times(T_k,\infty)$,}\\
u^\star =u_k & \mbox{on $\Omega\times \{t=T_k\}$},\\
\sigma^\star_{k+1}(u^\star)u^\star_x=0 & \mbox{on $\partial\Omega\times(T_k,\infty)$,}
\end{cases}
\end{equation}
admits a global solution $u^\star_{k+1}\in C^{2,1}(\bar\Omega\times[T_k,\infty);[0,1])$ with $u^\star_{k+1}\in C^{2+a,1+\frac{a}{2}}(\bar\Omega\times[T_k,T];[0,1])$ for each $T>T_k$ such that
\begin{itemize}
\item[(i$_{k+1}$)] $
0\le\min_{\bar\Omega}u^\star_{k+1}(\cdot,s)\le u^\star_{k+1}(x,t)\le\max_{\bar\Omega}u^\star_{k+1}(\cdot,s)\le 1\;\;\forall x\in\bar\Omega,\,\forall t\ge s\ge T_{k};
$
\item[(ii$_{k+1}$)] $
\int_\Omega u^\star_{k+1}(x,t)\,dx=L \bar{u}_k\;\;\mbox{for all $t\ge T_k$};
$
\item[(iii$_{k+1}$)] $
\|u^\star_{k+1}(\cdot,t)-\bar{u}_k\|_{L^\infty(\Omega)}\to 0\;\;\mbox{as $t\to\infty$}.
$
\end{itemize}
Define
\[
u_{k+1}(x)=u^\star_{k+1}(x,T_{k+1})\;\;\forall x\in\bar\Omega;
\]
then $u_{k+1}\in C^{2+a}(\bar\Omega;[0,1])$ and $u_{k+1}'(0)=u_{k+1}'(L)=0.$ Note also from (ii$_{k+1}$) that $\bar{u}_{k+1}=\bar{u}_k$.

For every $k\in\N$ and $(x,t)\in\bar{\Omega}_{T_{k-1}}^{T_k}$, we define
\begin{equation}\label{proof-3.5-7}
v^\star_k(x,t)=\int_0^x u_k^\star(y,t)\,dy+\int_{T_{k-1}}^t (u_k^\star)_x(0,s)\,ds
\end{equation}
and
\[
w^\star_k(x,t)=\int_{T_{k-1}}^t \rho_k^\star(v^\star_x(x,s))\,ds+\int_0^x v_{k-1}(y)\,dy,
\]
where
\[
v_{k-1}(x):=\int_0^x u_{k-1}(y)\,dy\;\;\forall x\in\bar\Omega;
\]
then $z^\star_k:=(v^\star_k,w^\star_k)\in(C^{3,1}\times C^{4,2})(\bar\Omega_{T_{k-1}}^{T_k})$ satisfies
\begin{equation}\label{proof-3.5-8}
\begin{cases}
  (v^\star_k)_t=(\rho^\star_k((v^\star_k)_x))_x & \mbox{in $\Omega_{T_{k-1}}^{T_k}$,} \\
  (w^\star_k)_x=v^\star_k & \mbox{in $\Omega_{T_{k-1}}^{T_k}$,} \\
  (w^\star_k)_t=\rho^\star_k((v^\star_k)_x) & \mbox{in $\Omega_{T_{k-1}}^{T_k}$,} \\
  v^\star_k=v_{k-1} & \mbox{on $\Omega\times\{t=T_{k-1}\}$,} \\
  \sigma^\star_k((v^\star_k)_x)(v^\star_k)_{xx}=0 & \mbox{on $\partial\Omega\times(T_{k-1},T_k)$.}
\end{cases}
\end{equation}

Next, for every $k\in\N$, define
\begin{equation}\label{proof-3.5-9}
Q_k=\{(x,t)\in\Omega_{T_{k-1}}^{T_k}\,|\,s^-(r_{1k})<(v^\star_k)_x(x,t)<s^+(r_{2k})\}.
\end{equation}
From (\ref{proof-3.5-1}), (\ref{proof-3.5-3}) for $k=1$, (\ref{proof-3.5-4}), (\ref{proof-3.5-5}), (\ref{proof-3.5-7}) for $k=1$, and (\ref{proof-3.5-9}) for $k=1$, it follows that \underline{$Q_1$ is a nonempty open subset of $\Omega_{T_1}=\Omega^{T_1}_{T_0}$ with}
\[
\underline{\bar{Q}_1\cap(\Omega\times\{0\})\ne\emptyset\;\;\mbox{and}\;\; \bar\Omega\times\{T_1\}\subset \bar{Q}_1.}
\]
Also, from (\ref{proof-3.5-5}), (\ref{proof-3.5-6}) for $k\ge1$, (i$_{k+1}$) for $k\ge1$, (\ref{proof-3.5-7}) for $k\ge2$, and (\ref{proof-3.5-9}) for $k\ge2$, we have $Q_k= \Omega_{T_{k-1}}^{T_k}$ for all $k\ge2.$

Let $k\in\N$ and $\epsilon>0.$ Following the notations in section \ref{sec:generic} with $r_1=r_{1k}$, $r_2=r_{2k}$, $t_1=T_{k-1}$, and $t_2=T_k$ in the current case, note from (\ref{proof-3.5-3}), (\ref{proof-3.5-8}), and (\ref{proof-3.5-9}) that (\ref{subsolution-1}) holds in $Q=Q_k$ for $z^\star=z^\star_k$. Thus, applying Theorem \ref{thm:two-wall}, we obtain a function $z^{(k)}=z^{(k)}_\epsilon\in W^{1,\infty}(\Omega^{T_k}_{T_{k-1}};\R^2)$ satisfying the following, where $z^{(k)}=(v^{(k)},w^{(k)})$:
\begin{itemize}
\item[(1)] $z^{(k)}$ is a solution of inclusion (\ref{inclusion}) in $Q=Q_k$;
\item[(2)] $z^{(k)}=z^\star_k$\;\;on $\bar{\Omega}^{T_k}_{T_{k-1}}\setminus Q_k$($=\partial\Omega^{T_k}_{T_{k-1}}$ when $k\ge 2$);
\item[(3)]$\nabla z^{(k)}=\nabla z^\star_k$\;\;a.e. on $\Omega^{T_k}_{T_{k-1}}\cap\partial Q_k$($=\emptyset$ when $k\ge 2$);
\item[(4)] $\|z^{(k)}-z^\star_k\|_{L^\infty(\Omega^{T_k}_{T_{k-1}};\R^2)}<\epsilon;$
\item[(5)] $\|v^{(k)}_t-(v^\star_k)_t\|_{L^\infty(\Omega^{T_k}_{T_{k-1}})}<\epsilon;$
\item[(6)] for any nonempty open set $O\subset Q_k,$
\[
\underset{O}{\mathrm{ess\,osc}}\, v^{(k)}_x \ge d_k,
\]
where $d_k=\min_{[r_{1k},r_{2k}]}\omega_2-\max_{[r_{1k},r_{2k}]}\omega_1 =s^+(r_{1k})-s^-(r_{2k})>s^+_0-s^-_0>0.$
\end{itemize}
In particular, from (1), a.e. in $Q_k$,
\begin{equation}\label{proof-3.5-10}
\begin{cases}
  v^{(k)}_x\in[s^-(r_{1k}),s^-(r_{2k})]\cup[s^+(r_{1k}),s^+(r_{2k})],  \\
  w^{(k)}_t=\rho(v^{(k)}_x), \\
  w^{(k)}_x=v^{(k)}.
\end{cases}
\end{equation}

For each $k\in\N$ and $T_{k-1}\le t<T_k,$ define
\begin{equation}\label{proof-3.5-11}
u(\cdot,t)=v^{(k)}_x(\cdot,t)\;\;\mbox{a.e. in $\Omega$};
\end{equation}
then from (2), (3), (\ref{proof-3.5-7}) and (\ref{proof-3.5-10}), we have \underline{$u\in L^\infty(\Omega_\infty;[0,1])$} and
\[
u=u^\star_1\;\;\mbox{a.e. in $\Omega_{T_1}\setminus Q$.}
\]
In particular, \underline{(a) in Theorem \ref{thm:(iv)} holds.} Also, from (6), (\ref{proof-3.5-10}), and (\ref{proof-3.5-11}), \underline{(b) and} \underline{(c) in Theorem \ref{thm:(iv)} are satisfied.} Next, observe from (2), (\ref{proof-3.5-7}), and (\ref{proof-3.5-11}) that for $k\in\N$ and $T_{k-1}\le t<T_k,$
\[
\begin{split}
\int_\Omega u(x,t)\,dx= &\, \int_\Omega v^{(k)}_x(x,t)\,dx= v^{(k)}(L,t)-v^{(k)}(0,t)\\
= &\, v^\star_k(L,t)-v^\star_k(0,t)= \int_\Omega (v^\star_k)_x(x,t)\,dx = \int_\Omega u^\star_k(x,t)\,dx.
\end{split}
\]
So \underline{(d) in Theorem \ref{thm:(iv)} follows} from (ii$_{j}$) $(j=1,\ldots,k)$. From (\ref{proof-3.5-10}) and (\ref{proof-3.5-11}),
\[
\begin{split}
\|\mathrm{dist} & (u,\{s^+(r_0),s^-(r_0)\})\|_{L^\infty(\Omega\times(T_k,\infty))} \\
\le & \min\{s^+(r_{2(k+1)})-s^+(r_{1(k+1)}),s^-(r_{2(k+1)})-s^-(r_{1(k+1)})\} \to 0\;\;\mbox{as $k\to\infty$.}
\end{split}
\]
\underline{This convergence implies (e) in Theorem \ref{thm:(iv)}.}

To check that $u$ is a global weak solution to problem (\ref{ib-P}), choose any $T>0$ and test function $\varphi\in C^\infty(\bar\Omega\times[0,T])$ with
\begin{equation*}
\varphi=0\;\;\mbox{on $\bar\Omega\times\{t=T\}$}\;\;\mbox{and}\;\;\varphi_x=0\;\;\mbox{on $\partial\Omega\times[0,T]$};
\end{equation*}
then $T_{k_0-1}<T\le T_{k_0}$ for some $k_0\in\N.$
We have to show that equality (\ref{eq:weaksol}) holds. Since the case that $T\le T_1$ can be handled similarly as in the proof of Theorem \ref{thm:(ii)-1}, we only deal with the case that $T>T_1,$ that is, $k_0\ge2.$

First, we decompose
\[
\begin{split}
\int_0^T\int_0^L(u\varphi_t &+\rho(u)\varphi_{xx})\,dxdt \\
&  = \bigg(\sum_{k=1}^{k_0-1}\int_{T_{k-1}}^{T_k}\int_0^L+ \int_{T_{k_0-1}}^T\int_0^L\bigg) (u\varphi_t+\rho(u)\varphi_{xx})\,dxdt \\
& =: \sum_{k=1}^{k_0-1} I_k+I_{k_0}.
\end{split}
\]
From Proposition \ref{prop:classical-global} and (\ref{proof-3.3-4}) with its following remark, we have
\[
I_{k_0}=-\int_{0}^L u_{k_0 -1}(x)\varphi(x,T_{k_0 -1})\,dx.
\]
On the other hand, proceeding as in the proof of Theorem \ref{thm:(ii)-2}, we obtain that for $k=1,\ldots,k_0-1,$
\[
I_k=\int_0^L (u_k(x)\varphi(x,T_k)-u_{k-1}(x)\varphi(x,T_{k-1}))\,dx.
\]
Thus,
\[
\sum_{k=1}^{k_0-1} I_k+I_{k_0}=-\int_{0}^L u_{0}(x)\varphi(x,T_{0})\,dx =-\int_{0}^L u_{0}(x)\varphi(x,0)\,dx;
\]
hence, \underline{$u$ is a global weak solution to (\ref{ib-P}).}

Finally, observe from (4) $(k\in\N)$ that \underline{there are infinitely many global weak} \underline{solutions to problem (\ref{ib-P}) satisfying properties (a)--(e) in Theorem \ref{thm:(iv)}.}

The proof of Theorem \ref{thm:(iv)} is now complete.
\end{proof}

\subsection*{Acknowledgments}
H. J. Choi was supported by the National Research Foundation of Korea (grant RS-2023-00280065) and Education and Research Promotion program of KOREATECH in 2024.
Y. Koh was supported by the National Research Foundation of Korea (grant NRF-2022R1F1A1061968).
This research was supported by Kyungpook National University Research Fund, 2023.

\subsubsection*{\emph{\textbf{Ethical Statement.}}} The manuscript has not been submitted to more than one journal for simultaneous consideration.  The manuscript has not been published previously.


\subsubsection*{\emph{\textbf{Conflict of Interest:}}}  The authors declare that they have no conflict of interest.

\end{document}